\documentclass[reqno]{amsart}
\usepackage{amscd,amssymb,mathrsfs}
\usepackage[cmtip,arrow,curve,matrix]{xy}
\DeclareFontFamily{OMS}{rsfs}{\skewchar\font'60}
\DeclareFontShape{OMS}{rsfs}{m}{n}{<-5>rsfs5 <5-7>rsfs7 <7->rsfs10 }{}
\DeclareSymbolFont{rsfs}{OMS}{rsfs}{m}{n}
\DeclareSymbolFontAlphabet{\scr}{rsfs}

\newcommand{\splitpageyesno}[2]{#2}

\splitpageyesno{
\usepackage{splitpage}
}
{
\usepackage{fullpage}
}

\numberwithin{equation}{section}
\newtheorem{Theorem}[equation]{Theorem}

\newtheorem{Proposition}[equation]{Proposition}

\theoremstyle{definition}

\newtheorem{Definition}[equation]{Definition}
\newtheorem{Remark}[equation]{Remark}
\newtheorem{Example}[equation]{Example}

\newcounter{thmItem}

\newcounter{textItem}

\newcounter{condItem}

\newcommand{\ainfty}{A_{\infty}}

\DeclareMathOperator{\Alg}{Alg}
\DeclareMathOperator{\CommAlg}{CommAlg}
\DeclareMathOperator{\Aut}{Aut}
\newcommand{\anomaly}[1]{\mathcal{A} (#1)}

\newcommand{\Bstring}{B\Gstring}
\newcommand{\BSpinc}{BSpin^{c}}
\newcommand{\Bspinc}{\BSpinc}
\newcommand{\Spinc}{Spin^{c}}

\newcommand{\CatOf}[1]{(\text{#1})}

\newcommand{\C}{\mathbb{C}}

\newcommand{\cochars}{\check{T}}

\DeclareMathOperator*{\colim}{colim}

\newcommand{\CQ}{C_{q}}

\newcommand{\dfcc}{\phi}

\newcommand{\einfty}{E_{\infty}}

\newcommand{\eqdef}{\overset{\text{def}}{=}}

\DeclareMathOperator{\End}{End}

\newcommand{\Fred}{\mathcal{F}}

\DeclareMathOperator{\Fun}{Fun}

\newcommand{\glsym}{gl_{1}}
\newcommand{\GLsym}{GL_{1}}
\newcommand{\gl}[1]{\glsym #1}
\newcommand{\GL}[1]{\GLsym #1}

\newcommand{\gpoid}[2]{
\xymatrix{
{#1}
 \ar@<.5ex>[r]
 \ar@<-.5ex>[r]
&
{#2}
}
}

\newcommand{\Gstring}{String}

\DeclareMathOperator{\ho}{ho}
\newcommand{\heq}{\simeq}
\newcommand{\htp}{\simeq}

\newcommand{\Hilb}{\mathcal{H}}
\DeclareMathOperator*{\holim}{holim}
\DeclareMathOperator{\Ho}{Ho}

\newcommand{\HTT}[1]{[HTT, #1]}

\renewcommand{\i}{\infty}

\newcommand{\iso}{\cong}

\newcommand{\linf}{\Omega^{\infty}}
\newcommand{\locsys}[1]{\{#1\}}

\newcommand{\Line}[1]{\mathrm{Line}_{{#1}}}

\newcommand{\lsb}[1]{( \! (#1) \! )}

\renewcommand{\L}{\mathscr{L}}

\DeclareMathOperator{\map}{map}

\newcommand{\Mod}[1]{\mathrm{Mod}_{{#1}}}

\newcommand{\mstring}{M\Gstring}

\DeclareMathOperator{\orient}{orient}

\renewcommand{\O}{\mathcal{O}}
\newcommand{\op}{\mathrm{op}}

\DeclareMathOperator{\Pic}{Pic}

\newcommand{\plus}{+}

\newcommand{\pt}{+}

\newcommand{\ptit}[1]{#1_{\pt}}
\newcommand{\ptspace}{*}

\newcommand{\Rmod}{\mbox{$\mathrm{R}$-$\mathrm{mod}$}}

\newcommand{\red}[1]{\widetilde{#1}}

\newcommand{\R}{\mathbb{R}}

\DeclareMathOperator{\spectra}{\mathscr{S}}
\DeclareMathOperator{\spaces}{\mathscr{T}}

\newcommand{\Un}{\mathrm{Un}}

\newcommand{\Set}{\mathrm{Set}}
\newcommand{\slot}{\,-\,}
\newcommand{\splus}{\Sigma^{\infty}_{\pt}}
\newcommand{\sinf}{\Sigma^{\infty}}
\DeclareMathOperator{\Sing}{Sing}

\DeclareMathOperator{\spf}{spf}

\newcommand{\Smash}{\wedge}

\DeclareMathOperator{\Stab}{Stab}
\newcommand{\Sp}{\mathrm{Sp}}

\newcommand{\T}{\mathbb{T}}
\newcommand{\TMF}{tmf}
\newcommand{\Tors}[1]{\mathrm{Tors} (#1)}
\DeclareMathOperator{\Tor}{Tor}

\newcommand{\Triv}[1]{\mathrm{Triv} (#1)}

\newcommand{\uln}[1]{\underline{#1}}

\newcommand{\UnivA}{\mathscr{M}}
\newcommand{\UnivL}{\mathscr{L}}

\newcommand{\xra}[1]{\xrightarrow{#1}}

\newcommand{\bN}{\mathbb{N}}
\newcommand{\Z}{\mathbb{Z}}

\begin{document}

\title{Twists of $K$-theory and $TMF$}
\author[Ando]{Matthew Ando}
\address{Department of Mathematics \\
The University of Illinois at Urbana-Champaign \\
Urbana IL 61801 \\
USA} \email{mando@math.uiuc.edu}
\thanks{M.~Ando was supported in part by NSF grant DMS-0705233.}
\author[Blumberg]{Andrew J. Blumberg}
\address{Department of Mathematics, University of Texas,
Austin, TX \ 78703}
\email{blumberg@math.utexas.edu}
\thanks{A.~J.~Blumberg was supported in part by NSF grant DMS-0906105}
\author[Gepner]{David Gepner}
\address{Department of Mathematics, The University of Illinois at Chicago,
Chicago IL 60607 \\ USA}
\email{gepner@uic.edu}
\thanks{}

\begin{abstract}
We explore an approach to twisted generalized cohomology from the
point of view of stable homotopy theory and $\i$-category theory
provided by \cite{0810.4535v3}.  We explain the relationship to the
twisted $K$-theory provided by Fredholm bundles.  We show how this
approach allows us to twist elliptic cohomology by degree four
classes, and more generally by maps to the four-stage Postnikov system
$BO\langle 0\dots 4\rangle.$ We also discuss Poincar\'e duality and
umkehr maps in this setting.
\end{abstract}

\maketitle

\tableofcontents

\section{Introduction}
\label{sec:introduction}

In \cite{0810.4535v3}, we and our co-authors generalize the classical
notion of Thom spectrum.  Let $R$ be an $A_{\infty}$ ring spectrum:
it has a space of units $\GL{R}$ which deloops to give a classifying
space $B\GL{R}.$  To a space $X$ and a map
\[
   \xi\colon X \to B\GL{R}
\]
we associate an $R$-module Thom spectrum $X^{\xi}$.  Letting $S$
denote the sphere spectrum, one finds  that $B\GL{S}$ is the classifying space
for stable spherical fibrations of virtual rank $0$, and $X^{\xi}$ is
equivalent to the classical Thom spectrum of the spherical fibration
classified by $\xi$ (as in, for example, \cite{LMS:esht}).

We remark in the introduction to \cite{0810.4535v3} that $B\GL{R}$
classifies the \emph{twists} of $R$-theory.  More
precisely, we define the $\xi$-twisted $R$-homology of $X$ to be
\[
     R_{k} (X)_{\xi} \eqdef \pi_{0}\Rmod (\Sigma^{k}R,X^{\xi}) \iso  \pi_{k}X^{\xi}
\]
and the $\xi$-twisted cohomology to be
\[
     R^{k} (X)_{\xi} \eqdef  \pi_{0}\Rmod (X^{\xi},\Sigma^{k}R),
\]
where here $\Sigma^{k}$ denotes the $k$-fold suspension, or
equivalently smashing with $S^k$.

In this paper, we expand on that remark, explaining how this definition
generalizes both singular cohomology with local coefficients and the
twists of $K$-theory studied by \cite{MR0282363,MR1018964,MR2172633}.
The key maneuver is to focus on the $\i$-categorical approach to Thom
spectra developed in \cite{0810.4535v3} (where by $\i$-categories we
mean the quasicategories of \cite{MR1935979,math.CT/0608040}).   We
show that the $\i$-category $\Line{R}$ of $R$-modules $L$ which admit
a weak equivalence $R\heq L$ is a model for $B\GL{R}$: we have a weak
equivalence of spaces
\[
    |\Line{R}| \heq B\GL{R}.
\]
One appeal of $\Line{R}$ is that, by construction, it classifies what one might call
``homotopy local systems'' of free rank-one
$R$-modules.  This flexible notion generalizes both classical local
coefficient systems and bundles of spaces (such as bundles of Fredholm
operators).  As one might expect, our work is closely related to the
parametrized spectra of May and Sigurdsson \cite{MR2271789}; we
discuss the relationship further in Section~\ref{s:param}.

As applications of our approach to twisted generalized cohomology, we
explain how the twisting of $K$-theory by degree three cohomology is
related to the $\Spinc$ orientation of Atiyah-Bott-Shapiro.
Similarly, recall that there is a map (unique up to homotopy)
\[
    BSpin \xra{\lambda} K (\Z,4)
\]
whose restriction to $BSU$ is the second Chern class.  The fiber of
$\lambda$ is called $BString$, and if $V$ is a $Spin$ vector bundle on
$X$, then a \emph{$String$} structure on $V$ is a trivialization of
$\lambda (V);$ that is, a map $g$ in the diagram
\[
\xymatrix{
&
{BString}
 \ar[d]^{\pi}\\
{X} \ar[r]^-{V}
 \ar@{-->}[ur]^{g}
&
{BSpin,}
}
\]
together with a homotopy $\pi g\Rightarrow V.$

The work of Ando, Hopkins, and Rezk \cite{AHR:orientation} constructs
an $\einfty$ $String$ orientation of $\TMF$, the spectrum of topological
modular forms.   (The discussion in this paper applies equally to the
connective spectrum $tmf$ and to the periodic spectrum $TMF$.)
Associating to a vector bundle its underlying
spherical fibration gives a map $BSpin \to B\GL{S},$ and associated to
the unit $S \to \TMF$ is a map $B\GL{S} \to B\GL{\TMF}.$  Composing
these, we have a map
\[
    k\colon BSpin \to B\GL{\TMF}.
\]
We show that the $\einfty$ $String$ orientation of
$\TMF$ of
\cite{AHR:orientation} implies the following.

\begin{Theorem} \label{t-th-TMF-twists-intro}
\hspace{5 pt}
\begin{enumerate}
\item $\TMF$ admits twists by degree-four integral cohomology.  More
precisely, there is a map $h\colon K (\Z,4) \to B\GL{\TMF}$
making the diagram
\[
\begin{CD}
      BSpin @> k >> B\GL{\TMF} \\
      @V \lambda VV @VV = V \\
      K (\Z,4) @> h >> B\GL{\TMF}
\end{CD}
\]
commute up to homotopy.  Thus given a map
$z\colon X \to K (\Z,4)$ (representing a class in
$H^{4} (X,\Z)$) we can define
\[
      \TMF^{*} (X)_{z} \eqdef \TMF^{*} (X)_{hz}.
\]
\item If $V$ is a $Spin$-bundle over $X$, classified by
\[
       X \xra{V} BSpin,
\]
then a homotopy $h\lambda (V) \Rightarrow k (V)$ determines an
isomorphism
\[
      \TMF^{*} (X^{V}) \iso  \TMF^{*}(X)_{\lambda (V)}.
\]
of modules over $\TMF^{*} (X).$
\end{enumerate}
\end{Theorem}

Theorem \ref{t-th-TMF-twists-intro} is well-known to the experts (e.g.
Hopkins, Lurie, Rezk, and Strickland).  As the reader will see, with
the approach to twisted cohomology presented here, it is an immediate
consequence of the $String$ orientation.

We also explain how twisted generalized cohomology is related to
Poincar\'e duality.  We briefly describe some work
in preparation, concerning twisted umkehr maps in generalized cohomology.
As  special cases, we recover the twisted $K$-theory umkehr map
constructed by Carey and Wang, and we construct an umkehr map
in twisted elliptic cohomology.  As we explain, our interest in the
twisted elliptic cohomology umkehr map arose from conversations with
Hisham Sati.

Finally, we report two applications of twisted equivariant elliptic
cohomology: we recall a result (due independently to the first author
\cite{Ando:EllLG} and Jacob Lurie), relating twisted equivariant
elliptic cohomology to representations of loop groups, and we explain
work of the first author and John Greenlees, relating twisted equivariant
elliptic cohomology to the equivariant sigma orientation.


\begin{Remark}
If $E$ is a cohomology theory and $X$ is a space, then $E^{*} (X)$
will refer to the \emph{unreduced} cohomology.  If $Z$ is a spectrum,
then we write $E^{*} (Z)$ for the spectrum cohomology.   We write
$\splus$ for the functor
\[
   \splus\colon \CatOf{spaces} \xra{\text{disjoint basepoint}}
   \CatOf{pointed spaces} \xra{\sinf} \CatOf{spectra},
\]
so we have by definition
\[
    E^{*} (X) \iso E^{*} (\splus X),
\]
while the reduced cohomology is
\[
   \red{E}^{*} (X)\iso E^{*} (\sinf X).
\]
If $V$ is a vector bundle over $X$ of rank $r$,  then we write $X^{V}$
for its Thom \emph{spectrum}:
this is equivalent to the suspension spectrum of the Thom space,
so $E^{*} (X^{V})$ is  the reduced cohomology of the Thom space.  Thus
the Thom isomorphism, if it exists, takes the form
\[
   E^{*} (X) \iso E^{*+r} (X^{V}).
\]
\end{Remark}

\begin{Remark}
The $\i$-category $\Line{R}$ is not the largest category we could use
to construct twists of $R$-theory.  If $R$ is an $\einfty$ ring
spectrum, i.e. a commutative ring spectrum, then we could consider the
$\i$-category $\Pic (R)$, consisting of invertible $R$-modules:
$R$-modules $L$ for which there exists an $R$-module $M$ such that
$L\Smash_{R}M \heq R$.
\end{Remark}
\bigskip
\noindent{\bf Acknowledgments}

This paper is the basis for a talk given by the first author at
the CBMS conference on $C^{*}$-algebras, topology, and physics at Texas
Christian University in May 2009.  We thank the organizers, Bob Doran and Greg
Friedman, for the opportunity.   It is a pleasure to acknowledge
stimulating conversations with Alan Carey,  Dan Freed, and Hisham Sati
which directly influenced this write-up.

\section{Classical examples of twisted generalized cohomology}
\label{sec:class-exampl-twist}

\subsection{Geometric models for twisted $K$-theory}\label{sec:geom-models-twist}

Let $\Hilb$ be a complex Hilbert space, and let $\Fred$ be its space
of Fredholm operators.  Then $\Fred$ is a representing space for
$K$-theory: Atiyah showed \cite{MR40:2053} that
\[
    K (X) \iso \pi_{0}\map (X,\Fred) = \pi_{0}\Gamma (X\times \Fred\to
    X).
\]
Atiyah and Segal \cite{MR2172633} develop the following approach to
twisted $K$-theory.    The unitary group $U=U (\Hilb )$ of $\Hilb$ acts on the
space $\Fred$ of Fredholm operators by conjugation.   Associated to a
principal $PU$-bundle $P\to X$, then, we can form the bundle
\[
   \xi = P \times_{PU} \Fred \to X
\]
with fiber $\Fred$.  They define the $P$-twisted $K$-theory of $X$ to be
\[
    K (X)_{P} = \pi_{0}\Gamma (\xi\to X).
\]
Thus one twists $K (X)$ by $PU$-bundles over $X$;
isomorphism classes of these are classified by $\pi_{0}\map (X,BPU);$
as $BPU$ is a model for $K (\Z,3)$, we have have
$\pi_{0}\map (X,BPU) \iso H^{3} (X;\Z).$

We warn the reader that this summary neglects
important and delicate issues which Atiyah and Segal address with
care, for example concerning the choice of topology on $U$ and $PU$.

Another approach to twisted $K$-theory passes through algebraic
$K$-theory; again we neglect important operator-theoretic
matters, referring the reader to \cite{MR1018964,MR1756434,BCMMS}
for details.  Conjugation induces an action of $PU$ on
the algebra $\mathcal{K}$ of compact operators on $\Hilb.$  Thus from
the $PU$-bundle $P \to X$ we can form the bundle
$P\times_{PU} \mathcal{K}$.  Let $\mathcal{A} = \Gamma (P\times_{PU}
\mathcal{K} \to X).$  This is a (non-unital) $C^{*}$-algebra, and
\[
   K (X)_{P} \iso K (\mathcal{A}).
\]
If $P$ is the trivial bundle, then $\mathcal{A}\iso \map
(X,\mathcal{K})$, and we have isomorphisms
\[
   K (\mathcal{A})\iso K (\map (X,\C)) \iso K (X).
\]

Both of these approaches to twisted $K$-theory are based on the idea
that from a $PU$-bundle we can build a bundle of copies of the
representing space for $K$-theory, and both have had a number of
successes.  They demand a good deal of information about $K$-theory, and
they exploit features of models of $K$-theory which may not be
available in other cohomology theories.

May and Sigurdsson show how to implement the construction of Atiyah
and Segal in the setting of their theory of parametrized stable
homotopy theory \cite[\S 22]{MR2271789}.  Specifically, they give a
construction of certain twisted cohomology theories associated to
parametrized spectra, and explain how the Atiyah-Segal definition fits
into their framework.  However, the approach of May and Sigurdsson
also takes advantage of good features of known models for K-theory
which may not be available in other cohomology theories.

In this paper we explain another way to locate twisted $K$-theory in
stable homotopy theory.  Our constructions continue to demand a
good deal of $K$-theory, for example that it be an $\ainfty$ or
$\einfty$ ring spectrum,  but many generalized cohomology theories $E$
satisfy our demands, and so our approach works in those cases as well.

Our approach incorporates and generalizes the construction of Thom
spectra of vector bundles, and so clarifies standard results
concerning twisted $E$-theory, such as the relationship to the Thom
isomorphism and Poincar\'e duality.  It also generalizes the classical
notion of (co)homology with local coefficients, as we now explain.

\subsection{Cohomology with local coefficients} \label{sec:cohom-with-local}

Let $X$ be a space, and let $\Pi_{\leq 1} (X)$ be its fundamental
groupoid.   We recall that a \emph{local coefficient system} on $X$ is
a functor\footnote{Since the fundamental groupoid is a groupoid, it is
equivalent to consider covariant or contravariant functors.}
\[
      \locsys{A}\colon \Pi_{\leq 1} (X) \to \CatOf{Abelian groups}.
\]
Given a local system $\locsys{A}$ on $X$, we can form the twisted singular
homology $H_{*} (X;\locsys{A})$ and cohomology  $H^{*} (X;\locsys{A})$.

\begin{Example}\label{ex-Serre-fib-local-coeff-sys}
For example, if $\pi\colon E \to X$ is a Serre fibration, then associating
to a point $p\in X$ the fiber $F_{p}= \pi^{-1} (p)$ gives rise to a
representation
\[
   F_{\bullet}\colon   \Pi_{\leq 1} (X) \xra{} \Ho \CatOf{spaces}.
\]
(Here $\Ho$ denotes the homotopy category obtained by inverting the
weak equivalences.)
Applying singular cohomology in degree $r$ produces the local
coefficient system $\locsys{H^{r} (F_{\bullet})}$.
\end{Example}

\begin{Example}
If $V$ is a vector bundle over $X$ of rank $r$, then taking the
fiberwise one-point compactification $V^{\plus}$ provides a Serre
fibration, and so we have the local coefficient system
\begin{equation}\label{eq:44}
      \locsys{\red{H}^{r} (V^{\plus}_{\bullet});\Z}
\end{equation}
whose value at $p\in X$ is the  cohomology group $\red{H}^{r}
(V^{\plus}_{p};\Z).$
\end{Example}

From the Serre spectral
sequence it follows that there is an isomorphism
\[
 H^{*} (X;\locsys{\red{H}^{r} (V^{\plus}_{\bullet})})
\iso H^{*+r} (X^{V};\Z)
\]
between the twisted cohomology of $X$ with coefficients in the system
$\locsys{\red{H}^{r} (V^{\plus}_{\bullet})}$ and the cohomology of
the Thom spectrum of $V$.

An \emph{orientation} of $V$ is a trivialization of the local system
\eqref{eq:44}, that is, an isomorphism of functors
\[
   \locsys{\red{H}^{r} (V^{\plus}_{\bullet})} \iso \Z
\]
(where $\Z$ denotes the evident constant functor).  It follows
immediately that an orientation of $V$ determines a Thom
isomorphism
\[
      H^{*} (X;\Z) \iso H^{*+r} (X^{V};\Z).
\]

\section{Bundles of module spectra}
\label{sec:bundl-module-spectra}

\subsection{The problem}
\label{sec:bundles-spectra-problem}

Let $X$ be a space.  We seek a notion of local system of spectra $\xi$
on $X$, generalizing the bundles of Fredholm operators in
\S\ref{sec:geom-models-twist} and the local systems of
\S\ref{sec:cohom-with-local}.
In particular, if $E\to X$ is a Serre fibration as in Example
\ref{ex-Serre-fib-local-coeff-sys}, then the classical local system
$\locsys{H^{r} (F_{\bullet})}$ should arise from a bundle of spectra
$F_{\bullet}\Smash H\Z$ by passing to homotopy groups.

From this example, we quickly see that while it is reasonable to ask
$\xi$ to associate to each point $p \in X$ a spectrum $\xi_{p}$, it
is too much to expect to associate to a path $\gamma\colon I\to X$
from $p$ to $q$ an isomorphism of fibers; instead, we expect a
homotopy equivalence
\[
   \xi_{\gamma}\colon \xi_{p} \to \xi_{q}.
\]
Moreover, an (endpoint-preserving) homotopy of paths $H\colon \gamma\to \gamma'$
should give rise to a path
\[
    \xi_{H}\colon \xi_{\gamma} \to \xi_{\gamma'}
\]
in the space of homotopy equivalences from $\xi_{p}$ to $\xi_{q}$.

These homotopy coherence issues quickly lead one to consider
representations of not merely the fundamental groupoid $\Pi_{\leq 1} (X)$,
but the whole fundamental $\infty$-groupoid $\Pi_{\leq \infty} (X),$
that is, the singular complex $\Sing X$.  Quasicategories make it both natural and inevitable
to consider such representations.

\subsection{$\i$-categories from spaces and from simplicial model categories}
\label{sec:i-categories-model-cats}

Recall that a quasicategory is a simplicial set which has fillings for
all inner horns.  Thus one source of quasicategories is spaces.  If
$X$ is a space, then its singular complex $\Sing X$ is a Kan complex:
it has fillings for all horns.   From the point of view of
quasicategories, where $1$-simplices correspond to morphisms, this
means that all morphisms are invertible up to (coherent higher)
homotopy.  Thus Kan complexes may be identified with
``$\i$-groupoids''.

We also recall (from \HTT{Appendix A and 1.1.5.9})
how simplicial model categories give rise to
quasicategories.  This procedure is an important source of
quasicategories, and it provides intuition about how
quasicategories encode homotopy theory.

If $\mathscr{M}$ is a simplicial model category, then we can
define  $\mathscr{M}^{\circ}$ to be the full subcategory consisting of
cofibrant and fibrant objects.  The simplicial nerve of
$\mathscr{M}^{\circ}$,
\[
   \mathscr{C} = N \mathscr{M}^{\circ},
\]
is the quasicategory associated to $\mathscr{M}$.

By construction, $\mathscr{C}$ is the
simplicial set in which
\begin{enumerate}
\item  the vertices $\mathscr{C}_{0}$ are cofibrant-fibrant objects of
$\mathscr{M}$;
\item  $\mathscr{C}_{1}$ consists of maps
\[
     L \to M
\]
between cofibrant-fibrant objects;
\item $\mathscr{C}_{2}$ consists of diagrams (not necessarily commutative)
\[
\xymatrix{
{L}
 \ar[r]^{f}
 \ar[dr]_{h}
&
{M}
 \ar[d]^{g}
\\
&
{N,}
}
\]
together with a homotopy from $gf$ to $h$ in the mapping space
(simplicial set)
$\mathscr{M} (L,M)$;
\end{enumerate}
and so forth.

In particular, in $\mathscr{C}$ the equivalences correspond to weak
equivalences in $\mathscr{M}^{\circ}$, that is, homotopy
equivalences.  Thus we may sometimes refer to the equivalences in a
quasicategory as weak equivalences or homotopy equivalences.

A simplicial model category $\mathscr{M}$ has an associated homotopy
category $\ho \mathscr{M}$, and an $\i$-category $\mathscr{C}$ has a
homotopy category $\ho \mathscr{C}$.  As one would expect, there is an
equivalence of categories (enriched over the homotopy category of
spaces)
\[
     \ho\mathscr{M} \heq \ho N \mathscr{M}^{\circ}.
\]
By analogy to the model category situation, if $\mathscr{C}$ is a
quasicategory and $\ho\mathscr{C} \heq \mathscr{D}$, then we shall say
that $\mathscr{C}$ is a ``model for $\mathscr{D}$''.

\subsection{The $\i$-category of $A$-modules}
\label{sec:i-cat-a-mod}

Let $\spectra$ be a symmetric monoidal $\i$-category of spectra.
Lurie constructs such an $\i$-category from scratch (\cite{DAGI} introduces
an  $\i$-category of spectra, which is shown to be monoidal in
\cite{DAGII}, and symmetric monoidal in \cite{DAGIII}).  Lurie shows
that his $\i$-category is equivalent to the symmetric monoidal
$\i$-category arising from the symmetric spectra of \cite{MR1695653},
and so by \cite{MR1806878} it is equivalent to the
symmetric monoidal $\i$-categories of spectra arising from various
classical symmetric monoidal simplicial model categories of spectra.
Let $S$ be the sphere spectrum.

\begin{Definition}
An \emph{$S$-algebra} is a monoid (strictly speaking, an algebra, since the relevant monoidal structure is not given by the cartesian product) in
$\spectra$.  We write $\Alg (S)$ for the $\i$-category of
$S$-algebras, and $\CommAlg (S)$ for the $\i$-category of commutative
$S$-algebras.
\end{Definition}

Using \cite{DAGII,DAGIII} and \cite{MR1806878} as above, one learns that
the symmetric monoidal structure on
$\spectra$ is such that $\Alg (S)$ is a model for $\ainfty$ ring
spectra, and $\CommAlg (S)$ is a model
for $\einfty$ ring spectra, so the reader is free to use his or
her favorite method to produce  $\i$-categories equivalent to $\Alg
(S)$ and $\CommAlg (S).$

\begin{Definition}\label{def-Mod-A}
If $A$ is an $S$-algebra, we let $\Mod{A}$ be the $\i$-category of
$A$-modules.
\end{Definition}

\begin{Example}
An $S$-module is just a spectrum, and
so $\Mod{S}$ is the $\i$-category of spectra.
\end{Example}

\subsection{Bundles of spaces and spectra}\label{s:param}

The purpose of this section is to introduce the $\i$-categorical model
of parametrized spectra we work with in this paper and compare it to
the May-Sigurdsson notion of parametrized spectra \cite{MR2271789}.

We begin by reviewing some models for the $\i$-category of spaces over
a cofibrant topological space $X$.  On the one hand, we have the
topological model category $\mathscr{T}/X$ of spaces over $X$,
obtained from the topological model category of spaces by forming the
slice category; i.e., the weak equivalences and fibrations are
determined by the forgetful functor to spaces.  We will refer to this
model structure as the ``standard'' model structure on
$\mathscr{T}/X$.  This is Quillen equivalent to the corresponding
simplicial model category structure on simplicial sets over $\Sing X$,
which in turn is Quillen equivalent to the simplicial model category of
simplicial presheaves on the simplicial category $\mathfrak{C}[\Sing
  X]$ (with, say, the projective model structure) \cite[\S
  2.2.1.2]{math.CT/0608040}.  Here $\mathfrak{C}$ denotes the left
adjoint to the simplicial nerve; it associates a simplicial category
to a simplicial set \cite[\S 1.1.5]{math.CT/0608040}.

\begin{Remark}
The Quillen equivalence between simplicial presheaves
and parametrized spaces depends on the fact that the base is an $\i$-groupoid
(Kan complex) as opposed to an $\i$-category; there is a more general theory of ``right fibrations'' (and, dually, ``left fibrations''),
but over a Kan complex a right fibration is a left fibration (and conversely) and therefore a Kan fibration.
\end{Remark}

On the level of $\i$-categories, this yields an equivalence
\[
\mathrm{St}\colon\mathrm{N}\Set_{\Delta/\Sing X}^\circ\longrightarrow\Fun(\Sing X^{\op},\mathrm{N}\Set_\Delta^\circ);
\]
the map, called the {\em straightening} functor, rigidifies a
fibration over $\Sing X$ into a presheaf of $\i$-groupoids on $\Sing
X$ whose value at the point $x$ is equivalent to the fiber over $x$
\cite[\S 3.2.1]{math.CT/0608040}.

A distinct benefit of the presheaf approach is a particularly
straightforward treatment of the base-change adjunctions.  Given a map
of spaces $f:Y\to X$, we may restrict a presheaf of $\i$-groupoids $F$
on $\Sing X$ to a presheaf of $\i$-groupoids $f^*F$ on $\Sing Y$.
This gives a functor, on the level of $\i$-categories, from spaces
over $X$ to spaces over $Y$, such that the fiber of $f^*F$ over the
point $y$ of $Y$ is equivalent to the fiber of $F$ over $f(y)$.
Moreover, $f^*$ admits both a left adjoint $f_!$ and a right adjoint
$f_*$, given by left and right Kan extension along the map $\Sing
Y^{\op}\to\Sing X^{\op}$, respectively.  Note that this is left and
right Kan extension in the $\i$-categorical sense, which amounts to
homotopy left and right Kan extension on the level of simplicial
categories or model categories.  On the level of model categories of
presheaves, there is an additional subtlety:
\[
f^*:\Fun(\mathfrak{C}[\Sing X^{\op}],\Set_\Delta)\longrightarrow\Fun(\mathfrak{C}[\Sing Y^{\op}],\Set_\Delta)
\]
is a {\em right} Quillen functor for the {\em projective} model
structure, with (derived) left adjoint $f_!$, and a {\em left} Quillen
functor for the {\em injective} model structure, with (derived) right
adjoint $f_*$, on the above categories of (simplicial) presheaves.
Of course the identity adjunction gives a Quillen equivalence between
these two model structures, but nevertheless one is forced to switch
back and forth between projective and injective model structures if
one wishes to simultaneously consider both base-change adjunctions.

Now we may stabilize either of the equivalent $\i$-categories
\[
\mathrm{N}\Set_{\Delta/\Sing X}^\circ\simeq\Fun(\Sing X^{\op},\mathrm{N}\Set_\Delta^\circ)
\]
by forming the
$\i$-category of spectrum objects in $\mathscr{C}_*$; here
$\mathscr{C}_*$ denotes the $\i$-category of {\em pointed} objects in
$\mathscr{C}$.  If $\mathscr{C}$ is an $\i$-category with finite
limits, then so is $\mathscr{C}_*$, and $\Stab(\mathscr{C})$ is
defined as the inverse limit of the tower
\[
\Stab(\mathscr{C})=\lim\{\cdots\overset{\Omega}{\longrightarrow}\mathscr{C}_*\overset{\Omega}{\longrightarrow}\mathscr{C}_*\}
\]
associated to the loops endomorphism
$\Omega\colon\mathscr{C}_*\to\mathscr{C}_*$ of $\mathscr{C}_*$.  In
other words, a spectrum object in an $\i$-category $\mathscr{C}$ (with
finite limits) is a sequence of pointed objects $A=\{A_0,A_1,\ldots\}$
together with equivalences $A_n\simeq\Omega A_{n+1}$ for each natural
number $n$.

Thus, our category of parametrized spectra is the stabilization
$\Stab(\mathrm{N}(\Set_\Delta/\Sing X)^\circ)$.  For our purposes, it
turns out to be much more convenient to use the presheaf model; there is
an equivalence of $\i$-categories
\[
\Stab(\mathrm{N}(\Set_\Delta/\Sing X)^\circ) \htp \Fun(\Sing
X^{\op},\spectra).
\]
Note that a functor $F\colon \Sing X\to\spectra$ associates to each
point $x$ of $X$ a spectrum $F_x$, to each path $x_0\to x_1$ in $X$ a
map of spectra (necessarily a homotopy equivalence) $F_{x_0}\to
F_{x_1}$, and so on for higher-dimensional simplices of $X$.

Given a presentable $\i$-category $\mathscr{C}$, the stabilization $\Stab(\mathscr{C})$ is itself presentable, and the functor $\Omega^\infty:\Stab(\mathscr{C})\to\mathscr{C}$ admits a left adjoint $\Sigma^\i:\mathscr{C}\to\Stab(\mathscr{C})$ \cite[Proposition 15.4]{DAGI}.
Just as $\Omega^\infty$ is natural in presentable $\i$-categories and {\em right} adjoint functors, dually, $\Sigma^\infty$ is natural in presentable $\i$-categories and {\em left} adjoint functors \cite[Corollary 15.5]{DAGI}.
In particular, given a map of spaces $f:Y\to X$, the adjoint pairs $(f_!,f^*)$ and $(f^*,f_*)$ defined above extend to the stabilizations, yielding a restriction functor
\[
f^*:\Fun(\Sing X^{\op},\spectra)\longrightarrow\Fun(\Sing Y^{\op},\spectra)
\]
which admits a left adjoint $f_!$ and a right adjoint $f^*$, again given by left and right Kan extension, respectively.

We can also formally stabilize suitable model categories, using
Hovey's work on spectra in general model categories \cite{Hovey}.
Specifically, given a left proper cellular model category
$\mathscr{C}$ and an endofunctor of $\mathscr{C}$, Hovey constructs a
cellular model category $\Sp^{\bN} \mathscr{C}$ of spectra.  When
$\mathscr{C}$ is additionally a simplicial symmetric monoidal
model category, the endofunctor given by the tensor with $S^{1}$
yields a simplicial symmetric monoidal model category of symmetric
spectra $\Sp^{\Sigma} \mathscr{C}$ (as well as a simplicial model
category $\Sp^{\bN} \mathscr{C}$ of prespectra).  These models of the
stabilization are functorial in left Quillen functors which are
suitably compatible with the respective endofunctors (see
\cite[5.2]{Hovey}).

In order to compare our model of parametrized spectra over $X$ to the
May-Sigurdsson model, we use the following consistency result.

\begin{Proposition}\label{prop:nstabcomm}
Let $\mathscr{C}$ be a left proper cellular simplicial model category
and write $\Sp^{\bN} \mathscr{C}$ for the cellular simplicial model
category of spectra generated by the tensor with $S^1$.  Then there is
an equivalence of $\i$-categories
\[
\mathrm{N}(\Sp^{\bN}
\mathscr{C})^\circ\simeq\Stab(\mathrm{N}\mathscr{C}^\circ).
\]
\end{Proposition}

\begin{proof}
The functors $\mathrm{Ev}_n\colon\Sp^{\bN}\mathscr{C} \to\mathscr C$,
which associate to a spectrum its $n^\mathrm{th}$-space $A_n$, induce
a functor (of $\i$-categories)
\[
f\colon\mathrm{N}(\Sp^{\bN}\mathscr{C})^\circ\to\lim\{\cdots\overset{\Omega}{\longrightarrow}\mathrm{N}\mathscr{C}^\circ_*\overset{\Omega}{\longrightarrow}
\mathrm{N}\mathscr{C}^\circ_*\}\simeq\Stab(\mathrm{N}\mathscr{C}^\circ)
\]
which is evidently essentially surjective.
To see that it is fully faithful, it suffices to check that for
cofibrant-fibrant spectrum objects $A$ and $B$ in $\Sp^{\bN}
\mathscr{C}$, there is an equivalence of mapping spaces
\[
\map(A,B)\simeq\holim\{\cdots\overset{\Omega}{\longrightarrow}\map(A_1,B_1)\overset{\Omega}{\longrightarrow}\map(A_0,B_0)\},
\]
where $\Omega\colon\map(A_{n+1},B_{n+1})\to\map(A_n,B_n)$ sends $A_{n+1}\to B_{n+1}$ to $A_n\simeq\Omega A_{n+1}\to\Omega B_{n+1}\simeq B_n$.
Since any cofibrant $A$ is a retract of a cellular object, inductively
we can reduce to the case in which $A=F_m X$, i.e., the shifted
suspension spectrum on a cofibrant object $X$ of $\mathscr{C}_*$.  Then
$\map(A,B)\simeq\map(X,B_m)$ by adjunction.  The latter is in turn
equivalent to $\map(\Sigma^{n-m} X, B_n)$, where we
interpret $\Sigma^{n-m} X=\ast$ for $m>n$, in which case the homotopy
limit is equivalent to that of the homotopically constant (above
degree $n$) tower whose $n^\mathrm{th}$ term is $\map(\Sigma^{n-m} X,
B_n)$.
\end{proof}

We now recall the May-Sigurdsson setup.  Given a space $X$, let
$(\mathscr{T}/X)_*$ denote the category of spaces over and under
$X$ ({\em ex-spaces}).  Although this category has a model structure
induced by the standard model structure on $\mathscr{T}/X$, one of the
key insights of May and Sigurdsson is that for the purposes of
parametrized homotopy theory it is essential to work with a variant
they call the $qf$-model structure \cite[6.2.6]{MR2271789}.  This
model structure is Quillen equivalent to the standard model structure
on ex-spaces \cite[6.2.7]{MR2271789}; however, its cofiber and fiber
sequences are compatible with classical notions of cofibration and
fibration (described in terms of extension and lifting properties).

May and Sigurdsson then construct a stable model structure on the
categories $\mathscr{S}_{X}$ of orthogonal spectra in $(\mathscr{T}/X)_*$
\cite[12.3.10]{MR2271789}.  This model structure is based on the
$qf$-model structure on ex-spaces, leveraging the diagrammatic
viewpoint of \cite{MR1806878, MandellMay}.  Similarly, they construct
a stable model structure on the category $\mathscr{P}_{X}$ of
prespectra in $(\mathscr{T}/X)_*$; the forgetful functor
$\mathscr{S}_{X} \to \mathscr{P}_{X}$ is a Quillen
equivalence \cite[12.3.10]{MR2271789}.

Using \cite[12.3.14]{MR2271789}, we see that after passing to
$\i$-categories the category $\mathscr{P}_X$ is in turn equivalent to
the category $\Sp^{\bN} (\mathscr{T}/X)_*$; the formal stabilization
of the $qf$-model structure on $(\mathscr{T}/X)_*$ with respect to the
fiberwise smash with $S^1$.  Using Proposition~\ref{prop:nstabcomm}
and the fact that the $qf$-model structure is Quillen equivalent to
the standard model structure, we obtain equivalences of
$\i$-categories
\begin{align*}
\mathrm{N}(\mathscr{S}_X)^{\circ} & \to \mathrm{N}(\mathscr{P}_X)^{\circ} \to
\mathrm{N}(\mathrm{Sp}^{\mathscr{\bN}}(\mathscr{T}/X)_*)^\circ \\
& \to \Stab (\mathrm{N}(\mathscr{T}/X)^\circ)\to\Stab(\mathrm{N}(\Set_\Delta/\Sing
X)^\circ)\to\Fun(\Sing X^{\op},\spectra).
\end{align*}

Thus we obtain the following comparison theorem.

\begin{Theorem}\label{thm:comp}
There is an equivalence of $\i$-categories between the simplicial
nerve of the May-Sigurdsson category of parametrized orthogonal
spectra $\mathrm{N}(\mathscr{S}_X)^{\circ}$ and the $\i$-category
$Fun(\Sing X^{\op}, \spectra)$ of presheaves on $X$ with values in
spectra.
\end{Theorem}

Furthermore, the derived base-change functors we construct via the
stabilization of the presheaves agree with the derived base-change
functors constructed by May and Sigurdsson.  To see this, observe that
it suffices to check this for $f^*$; compatibility then follows
formally for the adjoints $f_*$ and $f_!$.  Moreover, since $f^*$ on
the categories of spectra is obtained as the suspension of $f^*$ on
spaces, we can reduce to checking that the right derived functor of
$f^* \colon (\mathscr{T}/X)_* \to (\mathscr{T}/Y)_*$ in the $qf$-model
structure is compatible with the right derived functor of $f^* \colon
\Fun(\mathfrak{C}(\Sing X^{\op}), \Set_\Delta) \to
\Fun(\mathfrak{C}(\Sing Y^{\op}), \Set_\Delta)$ in the projective
model structure.  By the work of \cite[\S 9.3]{MR2271789}, it suffices
to check the compatibility for $f^*$ in the $q$-model structure.
Since both versions of $f^*$ that arise here are Quillen right
adjoints, this amounts to the verification that the diagram
\[
\xymatrix{
\Fun(\mathfrak{C}(\Sing X^{\op}), \Set_\Delta) \ar[d]^{\Un} \ar[r]^-{f^*} &
\Fun(\mathfrak{C}(\Sing Y^{\op}), \Set_\Delta) \ar[d]^{\Un} \\ 
\Set_\Delta/X \ar[r]^-{f^*} & \Set_\Delta/Y  \\
}
\]
commutes when applied to fibrant objects, where here $\Un$ denotes the
unstraightening functor (which is the right adjoint of the Quillen
equivalence).  Finally, this follows from \HTT{2.2.1.1}.

\subsection{Bundles of $A$-modules and $A$-lines}
\label{sec:bundles-a-modules}

If $X$ is a space, let $\Sing X$ be its singular complex.  The work of
the previous section justifies the following definition.

\begin{Definition}\label{def-bundle-of-A-modules}
A \emph{bundle} or \emph{homotopy local system} of $A$-modules over
$X$ is a map of simplicial sets
\[
   f\colon \Sing X \to \Mod{A}.
\]
Similarly if $Y$ is any $\i$-groupoid, then a bundle of $A$-modules
over $Y$ is just a map of simplicial sets
\[
   f\colon Y \to \Mod{A}.
\]
\end{Definition}
Thus $f$ assigns
\begin{enumerate}\setcounter{enumi}{-1}
\item to each point $p\in X$ an $A$-module $f (p)$;
\item to each path $\gamma$ from $p$ to $q$ a map of
$A$-modules
\begin{equation}\label{eq:29}
     f (\gamma)\colon f (p) \to f (q);
\end{equation}
\item to each $2$-simplex $\sigma\colon \Delta^{2}\to X,$ say
\[
\xymatrix{
{p}
 \ar[d]_{\sigma_{01}}
 \ar[dr]^{\sigma_{02}}
\\
{q}
 \ar[r]_{\sigma_{12}}&
{r,}
}
\]
a path $f (\sigma)$ in $\Mod{A}(f (p), f (r))$ from $f (\sigma_{12})f
(\sigma_{01})$ to $f (\sigma_{02})$;
\end{enumerate}
and so forth.

Recall \HTT{1.2.7.3} that if $Y$ is a simplicial set and $\mathcal{C}$ is an
$\i$-category, then the simplicial mapping space
$\mathcal{C}^{Y}$ is the $\i$-category $\Fun(Y,\mathcal{C})$ of functors from $Y$ to $\mathcal{C}$.

\begin{Definition} \label{def-i-cat-bundle-A-modules}
The $\i$-category of bundles of
$A$-modules over $X$
is the simplicial mapping space
\[
\Mod{A}^{X} \eqdef \Fun (\Sing X,\Mod{A}).
\]
\end{Definition}


\begin{Remark}
We have not set up the framework necessary to work directly with the
bundle of $A$-modules associated to $f\colon \Sing X \to \Mod{A}$
(although see Theorem~\ref{thm:comp}).  Nonetheless, the notation of
bundle and pullback is compelling, and so we write $\UnivA$ for the
identity map
\[
  \Mod{A} \to \Mod{A},
\]
and if $f\colon \Sing X \to \Mod{A}$ is a map of $\i$-categories, then
we may write $f^{*}\UnivA$ as a synonym for $f$, when we want to
emphasize its bundle aspect.
\end{Remark}

Recall that $\Sing X$ is a Kan complex or $\i$-groupoid: it satisfies the
extension condition for all horns. Viewing an $\i$-category as a model
for a homotopy theory, an $\i$-groupoid models a homotopy theory  in
which all the morphisms are homotopy equivalences.

In particular, the map $f (\gamma)$ in \eqref{eq:29} is necessarily
an equivalence: the $A$-modules $f (p)$ will vary through weak
equivalences as $p$ varies over a path component of $X$.  We shall
be particularly interested in the case that these fibers are free
rank-one $A$-modules.

\begin{Definition}\label{def-line}
An \emph{$A$-line} is an $A$-module $L$ which admits a weak equivalence
\[
    L \xra{\heq} A.
\]
The $\i$-category $\Line{A}$ is the maximal $\i$-groupoid in $\Mod{A}$
generated by the $A$-lines.   We write $j$ for the inclusion
\[
    j\colon \Line{A} \to \Mod{A}
\]
and $\UnivL \eqdef j^{*}\UnivA$ for the tautological
bundle of $A$-lines over $\Line{A}.$
\end{Definition}

By construction $\Line{A}$ is a Kan complex, and we regard it
as the classifying space for bundles of $A$-lines.  If
$X$ is a space, then a map
\[
   f\colon \Sing X \to \Line{A}.
\]
assigns
\begin{enumerate}\setcounter{enumi}{-1}
\item to each point $p\in X$ an $A$-line $f (p)$;
\item to each path $\gamma$ from $p$ to $q$ an equivalence map of
$A$-lines
\[
     f (\gamma)\colon f (p) \heq f (q);
\]
\item to each $2$-simplex $\sigma\colon \Delta^{2}\to X,$ say
\[
\xymatrix{
{p}
 \ar[d]_{\sigma_{01}}
 \ar[dr]^{\sigma_{02}}
\\
{q}
 \ar[r]_{\sigma_{12}}&
{r,}
}
\]
a path $f (\sigma)$ in $\Line{A}(f (p), f (r))$ from $f (\sigma_{12})f
(\sigma_{01})$ to $f (\sigma_{02})$;
\end{enumerate}
and so forth.

\begin{Definition}
The simplicial mapping space $\Line{A}^{X} = \Fun (\Sing X,\Line{A})$
is an $\i$-category (in fact, a Kan complex); we call it the the
$\i$-category or space of \emph{$A$-lines over $X.$}
\end{Definition}

We develop twisted $A$-theory starting from
$\Line{A}$ in \S\ref{sec:twist-gener-cohom}.  Before doing so, we
briefly discuss other aspects of the $\i$-category $\Line{A}$.

\subsection{$\Line{A}$ and $\GL{A}$}

\label{sec:line_a-gla}

By construction, $\Line{A}$ is connected, and so
equivalent to the maximal $\i$-groupoid $B\Aut (A)$ on the single
$A$-module $A$.  As we discuss in \cite[\S6]{0810.4535v3}, it is an
important point
that the space of morphisms $\Aut (A)= \Line{A} (A,A)$ is not
a group, or even a monoid, but instead merely a group-like $A_{\infty}$
space.

Nevertheless, $\Line{A}$ is not only the classifying space
for bundles of $A$-lines, but it is a delooping of $\Aut (A).$  To see
this, let $\Triv{A}$ be the $\i$-category of $A$-lines $L$, equipped with
an equivalence $L \xra{\heq} A.$  Then
\cite[Prop. 7.38]{0810.4535v3} $\Triv{A}$ is
contractible, and the map
\[
    \Triv{A} \to \Line{A}
\]
is a Kan fibration, with fiber $\Aut (A)$.

Classical infinite loop space theory provides another model for homotopy type
$\Aut (A).$  Namely, let $A$ be an $\ainfty$ ring spectrum in the
sense of \cite{LMS:esht}: so $\pi_{0}\linf A$ is a ring.  Let $\GL{A}$
be the pull-back in the diagram
\[
\begin{CD}
\GL{A} @>>> \linf A \\
@VVV @VVV \\
(\pi_{0}\linf A)^{\times} @>>> \pi_{0}\linf A.
\end{CD}
\]
Then $\GL{A}$ is a group-like $\ainfty$ space: $\pi_{0}\GL{A}$ is a
group.  We show that
\[
    \GL{A} \heq |\Aut (A)|.
\]
Since the geometric realization of a Kan fibration is a Serre
fibration \cite{MR0238322}, the fibration
\[
     \Aut (A) \to \Triv{A} \to \Line{A}
\]
gives rise to a fibration
\begin{equation}\label{eq:30}
    \GL{A} \heq |\Aut (A)| \to |\Triv{A}| \heq \ptspace  \to |\Line{A}|.
\end{equation}
Thus $|\Line{A}|$ provides a model for the delooping
$B\GL{A}$.  It has the virtue that we have already given a precise
description of the vertices of the simplicial mapping space
\[
   \Line{A}^{X} = \map (\Sing X,\Line{A}) \heq \map (X,|\Line{A}|).
\]

\begin{Example}
If $S$ is the sphere spectrum, then $\linf S$ is the space $QS^{0} =
\linf \sinf S^{0}$, and
\[
\GL{S} = Q_{\pm 1}S^{0},
\]
i.e., the unit components.  The space $B\GL{S}\heq |\Line{S}|$ is the
classifying space for stable spherical fibrations of virtual rank $0$.
It follows the space of $S$-lines over $X$ is homotopy equivalent to
the space of spherical fibration of virtual rank $0$.
\end{Example}

\begin{Example}
The classical $J$-homomorphism is a map
\[
   J\colon O \to \GL{S},
\]
which deloops to give a map
\[
   BJ\colon BO \to B\GL{S}.
\]
One sees that this is the map which takes a virtual vector bundle
of rank $0$ to its associated stable spherical fibration; we may
regard this as associating to
a vector bundle its bundle of $S$-lines.
\end{Example}

\begin{Example}\label{ex-8}
If $A$ is an $S$-algebra, then the unit of $A$ induces a map
\[
   B\GL{S} \to B\GL{A}.
\]
In our setting, this map arises from the map
of $\i$-categories
\[
  \Mod{S} \to \Mod{A}
\]
given by $M \mapsto M\otimes_{S} A = M\Smash_{S}A$, which restricts to
give a map of $\i$-categories
\[
   \Line{S} \to \Line{A}.
\]
\end{Example}

\begin{Example}[\cite{MQRT:ersers}]\label{ex-1}
Let $H\Z$ be the integral Eilenberg-MacLane spectrum.  Then $\linf
H\Z\heq K (\Z,0)\heq \Z$, and so $\GL{H\Z} \heq \{\pm 1\}\heq \Z/2$, and
$B\GL{H\Z}\heq B\Z/2\heq K (\Z/2,1).$
\end{Example}

\begin{Remark}
If $A=K$, the spectrum representing complex $K$-theory, then $\Aut (K)$
has the homotopy type of the space of $K$-module equivalences $K\to
K.$    Atiyah and Segal \cite{MR2172633} build twists of $K$-theory
from $PU$-bundles.  They remark that one can more generally build
twists of $K$-theory from $\mathcal{G}$-bundles, where $\mathcal{G}$
is the group of strict $K$-module automorphisms of $K$-theory.  Our
space $\Aut (K)$ generalizes this idea.
\end{Remark}

\begin{Remark}
As we explain in \cite[\S6]{0810.4535v3}, for
many algebras $A$ (including the sphere $S$), the group $\Aut_{strict}
(A)$ of strict $A$-module automorphisms of $A$ cannot provide a
sufficiently rich theory of bundles of $A$-modules.  For example,
Lewis's Theorem \cite{Lewis} implies that there is no model for the
sphere spectrum $S$ such that the classifying space $B\Aut_{strict}
(S)$ classifies stable spherical fibrations.  (See also \cite[\S
  22.2]{MR2271789} for discussion of this issue.)
\end{Remark}


\section{The generalized Thom spectrum}
\label{sec:gener-thom-spectr}

Let $A$ be an $S$-algebra, let $X$ be a space, and let $f$ be a bundle
of $A$-lines over $X$, that is, a map of simplicial sets
\[
   f\colon \Sing X \to \Line{A}.
\]

Although the $\i$-category $\Line{A}$ is not cocomplete (it doesn't
even have sums), the $\i$-category $\Mod{A}$ of $A$-modules is
complete and cocomplete.    This allows us in \cite{0810.4535v3} to
make the following definition.

\begin{Definition}\label{def-Thom-spectrum}
The \emph{Thom spectrum} of $f$ is the colimit
\[
   X^{f} \eqdef \colim \left( \Sing X \xra{f} \Line{A} \xra{j} \Mod{A} \right).
\]
Equivalently, $X^{f}$ is the left Kan extension $L_{\pi}jf$ in the
diagram
\[
\xymatrix{
{X}
  \ar[r]^-{f}
  \ar[d]_-{\pi}
&
{\Line{A}}
 \ar[r]^-{j}
&
{\Mod{A}.}
\\
{\ptspace}
 \ar@{-->}[urr]_-{L_{p} (fj)}
}
\]
\end{Definition}

The colimit and left Kan extension here are $\i$-categorical
colimits: they are generalizations of the notion of homotopy colimit
and homotopy left Kan extension.  It is an important achievement of
$\i$-category theory to give a  sensible definition of these colimits.

Let $A_{X}\colon X \to \ptspace \to \Line{A}$ be the map which picks out $A$,
considered as the constant $A$-line over $X$.  The colimit means that we have
an equivalence of mapping spaces
\[
   \Mod{A} (X^{f},A) \heq \Mod{A}^{X} (f^{*}j^{*}\UnivA,A_{X}).
\]
Notice also that we have a natural inclusion
\begin{equation}\label{eq:31}
   \Line{A}^{X} (f^{*}\UnivL,A_{X}) \to \Mod{A}^{X} (f^{*}j^{*}\UnivA,A_{X}):
\end{equation}
a map of bundles of $A$-modules $f^{*}\UnivL \to A_{X}$ is a map of
bundles of $A$-lines
if it is an equivalence over every point of $X$, and one checks that
the inclusion \eqref{eq:31} is the inclusion of a set of path
components.

\begin{Definition}
The space of \emph{orientations} of $X^{f}$ is the pull-back in the
diagram
\begin{equation}
\begin{CD}
\orient (X^{f},A) @>>> \Mod{A}(X^{f},A) \\
@V\heq VV @VV \heq V \\
\Line{A}^{X} (f^{*}\UnivL,A_{X}) @>>> \Mod{A}^{X}(f^{*}j^{*}\UnivA, A_{X}).
\end{CD}
\end{equation}
That is, the space of orientations $\orient (X^{f},A)$ is the subspace
of $A$-module maps $X^{f} \to A$ which correspond, under the equivalence
\[
   \Mod{A} (X^{f},A) \heq \Mod{A}^{X} (f^{*}j^{*}\UnivA, A_{X}),
\]
to fiberwise equivalences $f^{*}\UnivL \to A_{X}.$
\end{Definition}

This appealing notion of orientation expresses
orientations as fiberwise equivalences of bundles of spectra.  The
following results from \cite{0810.4535v3} explain how our Thom spectra and
orientations generalize the classical notions.

\begin{Theorem}\label{t-th-i-cat-thom-is-thom}
Suppose that
\[
\xi\colon \Sing X \to \Line{S}
\]
corresponds to a map
\[
g\colon X \to B\GL{S}.
\]
Then $X^{\xi}$ is equivalent to the
classical Thom spectrum $X^{g}$ of the spherical fibration classified
by $g$.  It follows (see  Example \ref{ex-8}) that if $f$ is the composition
\[
   f\colon \Sing X \xra{\xi} \Line{S} \to \Line{A},
\]
then $X^{f} \heq X^{\xi}\Smash_{S} A \heq X^{g}\Smash_{S} A$ is
equivalent to the classical Thom spectrum tensored with $A.$
\end{Theorem}

We can then study the space of orientations of $X^{f}$ via the equivalences
\begin{equation} \label{eq:32}
    \Mod{A} (X^{f},A) \heq \Mod{A} (X^{g}\Smash_{S} A,A) \heq \Mod{S} (X^{g},A),
\end{equation}
and we find that

\begin{Proposition}
A map $\alpha\colon X^{f} \to A \in \Mod{A} (X^{f},A)$ is in $\orient (X^{f},A)$
if and only if it corresponds to an orientation $\beta\colon X^{g}\to A$ of the
classical Thom spectrum, that is, if and only if
\[
    (\ptit{X} \xra{z} A) \mapsto (X^{g} \xra{\Delta} \ptit{X}\Smash X^{g}
    \xra{z \Smash \beta} A \Smash A \to A)
\]
induces an isomorphism
\[
   A^{*} (\ptit{X}) \iso A^{*} (X^{g}).
\]
\end{Proposition}

Our theory leads to an obstruction theory for orientations.  Let
$\map_{f} (\Sing X,\Triv{A})$ be the simplicial set which is the
pull-back in the diagram
\[
\begin{CD}
\map_{f} (\Sing X,\Triv{A}) @>>> \map (\Sing X,\Triv{A}) \\
@VVV @VVV \\
\{f \} @>>> \map (\Sing X,\Line{A}).
\end{CD}
\]
That is, $\map_{f} (\Sing X,\Triv{A})$ is the mapping simplicial set of
lifts in the diagram
\begin{equation} \label{inf-ii-eq:26}
\xymatrix{
&
{\Triv{A}}
 \ar[d] \\
{\Sing X}
 \ar@{-->}[ur]
 \ar[r]_-{f}
&
{\Line{A}.}
}
\end{equation}

The obstruction theory for orientations of the bundle of $A$-modules
is given by the following.

\begin{Theorem} \label{t-th-inf-ii-or-thy-infty-lifting}
Let $f\colon \Sing X \to \Line{A}$ be a bundle of $A$-lines over $X$, and
let $X^{f}$ be the associated $A$-module Thom
spectrum.   Then there is an equivalence
\[
   \map_{f} (\Sing X,\Triv{A}) \heq
   \Line{A}^{X} (f,\iota) \heq \orient (X^{f},A).
\]
In particular, the bundle $f^{*}\UnivL$ admits an orientation if and
only if $f$ is null-homotopic.
\end{Theorem}

\begin{Example}
This theorem recovers and slightly generalizes the obstruction theory of
\cite{MQRT:ersers} (which treats the case that $A$ is a $\einfty$
ring spectrum, that is, a commutative $S$-algebra).
Let $g\colon X\to B\GL{S}$ be a stable spherical fibration.  Then $g$
admits a Thom isomorphism in $A$-theory if and only if the composition
\[
     X \xra{g} B\GL{S} \heq |\Line{S}| \to |\Line{A}| \heq B\GL{A}
\]
is null.
\end{Example}

\begin{Example}
This example appears in \cite{MQRT:ersers}.  Let $H\Z$ be the
integral Eilenberg-MacLane spectrum.  From Example \ref{ex-1} we have
$B\GL{H\Z}\heq K (\Z/2,1).$  The obstruction to orienting a vector
bundle $V/X$ in singular cohomology is the map
\[
   X \xra{V} BO \xra{BJ} B\GL{S} \xra{} B\GL{H\Z} \heq K (\Z/2,1);
\]
this is just the first Stiefel-Whitney class.
\end{Example}


\section{Twisted generalized cohomology}
\label{sec:twist-gener-cohom}

Now we consider twisted generalized cohomology in the language of
sections~\ref{sec:bundl-module-spectra} and
\ref{sec:gener-thom-spectr}.  Let $A$ be an $S$-algebra, and let
\[
     f\colon \Sing X\to \Line{A}
\]
or, equivalently, $f\colon X \to B\GL{A}$ (see \S\ref{sec:line_a-gla})
classify a bundle of $A$-lines over $X.$  As in Definition
\ref{def-Thom-spectrum}, let
\[
    X^{f} = \colim \left( \Sing X \xra{f} \Line{A} \xra{j} \Mod{A}\right)
\]
be the indicated $A$-module.  We think of $X^{f}$ as the $f$-twisted
cohomology object associated to the bundle $f$, and we make the following

\begin{Definition} \label{def-f-twisted-A-homology}
The \emph{$f$-twisted $A$ homology and cohomology groups} of $X$
are
\begin{align*}
    A^{n} (X)_{f} &\eqdef \pi_{0}\Mod{A} (X^{f}, \Sigma^{n}A) \\
    A_{n} (X)_{f} & \eqdef \pi_{0}\Mod{A} (\Sigma^{n}A, X^{f}).
\end{align*}
Equivalently, we have
\begin{align*}
    A^{n} (X)_{f} & = \pi_{-n}F_{A} (X^{f},A) \\
    A_{n} (X)_{f} & = \pi_{n}F_{A} (A,X^{f}) \iso \pi_{n}X^{f}.
\end{align*}
\end{Definition}

Here if $V$ and $W$ are $A$-modules, then  $F_{A} (V,W)$ is the
function spectrum of $A$-module maps from $V$ to $W$: it is a spectrum
such that
\[
     \linf F_{A} (V,W) \heq \Mod{A} (V,W).
\]
Thus for $n\geq 0$,
\[
    \pi_{n} F_{A} (V,W) \iso \pi_{n}\Mod{A} (V,W) \iso \Mod{A}
    (\Sigma^{n}V,W) \iso \Mod{A} (V,\Sigma^{-n}W).
\]

\begin{Example}\label{ex-9}
Suppose that $V$ is a vector bundle over $X$.  Then we can form the
map
\[
j (V)\colon X \xra{V} BO \xra{BJ} B\GL{S}  \xra{} B\GL{A}.
\]
and also the twisted cohomology
\[
    A^{*} (X)_{j (V)} = \pi_{0}\Mod{A} (X^{j (V)},\Sigma^{*}A).
\]
Since by Theorem \ref{t-th-i-cat-thom-is-thom}
\[
    X^{j (V)} \heq X^{V}\Smash A,
\]
we have
\begin{equation}\label{eq:28}
   A^{*} (X)_{j (V)} \iso \pi_{0} \Mod{S} (X^{V},\Sigma^{*}A) = A^{*} (X^{V}),
\end{equation}
so in this case the twisted cohomology is just the cohomology of the
Thom spectrum.
\end{Example}

The definition is not quite a direct generalization of that Atiyah and
Segal.  Let $S_{X}\colon \Sing X \to \Mod{S}$ be the constant functor which
attaches to each point of $X$ the sphere spectrum $S$.
Theorem~\ref{thm:comp} and the work of \cite[\S 22]{MR2271789} allow
us to describe their construction as attaching to the bundle of
$A$-lines $f^{*}\UnivL$ over $X$ the group
\[
    A (X)_{f,AS} = \pi_{0}\Gamma (f/X)\iso
  \pi_{0}\Mod{S}^{X} (S_{X},f^{*}\UnivL).
\]
To compare this definition to ours, we must assume (as is the case for
$K$-theory, for example) that $A$ is a \emph{commutative} $S$-algebra.
In that case, if $L$ is an $A$-module, then the dual spectrum
\[
L^{\vee} \eqdef F_{A} (L,A)
\]
is again an $A$-module.  As in the case of classical commutative
rings, the operation $L\mapsto L^{\vee}$ defines an involution
\[
   \Line{A} \xra{(\slot)^{\vee}} \Line{A}
\]
on the $\i$-category of $A$-lines, such that
\[
    L^{\vee}\Smash_{A} M \heq F_{A} (L,M);
\]
in particular
\[
    \Mod{S} (S,L^{\vee}) \heq \Mod{A} (A,L^{\vee}) \heq \Mod{A} (L,A).
\]

\begin{Remark}
As we have written it, we have an evident functor $\Line{A}\to
(\Line{A})^{\text{op}}$. As $\Line{A}$  is an $\i$-groupoid, it admits
an equivalence $\Line{A}\heq \Line{A}^{\text{op}}$.
\end{Remark}

This is an $\i$-categorical approach to the following: if $A$ is a commutative
$S$-algebra then $\GL{A}$ is a sort of commutative group.  More
precisely, it is a commutative group-like monoid in the $\i$-category
of spaces, or equivalently it is a group-like $\einfty$-space.  As
such it has an involution
\[
-1\colon\GL{A}\to \GL{A}
\]
which deloops to a map
\[
   B(-1)\colon B\GL{A} \to B\GL{A}.
\]

In any case, given a map
\[
    f\colon \Sing X \to  \Line{A},
\]
classifying the bundle $f^{*}\L$, we may form the map
\[
    -f = f^{\vee}\colon \Sing X \xra{f} \Line{A} \xra{(\slot)^{\vee}} \Line{A}
\]
so that $(-f)^{*}\L$ is the fiberwise dual of $f^{*}\L$,
and then one has
\begin{align*}
     \Gamma ((-f)^{*}\L) &= \Mod{S}^{X} (S_{X},(-f)^{*}\L) \\
                         &  \heq  \Mod{A}^{X} (A_{X},(-f)^{*}\L) \\
                         & \heq    \Mod{A}^{X} (f^{*}\L,A) \\
                         &  \heq \Mod{A} (X^{f},A).
\end{align*}
That is, the cohomology object we associate to $f\colon X\to B\GL{A}$ is the
one which Atiyah and Segal associate to $-f\colon X \to B\GL{A}$,
\[
    A^{*} (X)_{f} \iso A^{*} (X)_{-f,AS}.
\]

Of course Atiyah and Segal also explain how to construct a $K$-line
from a $PU$-bundle over $X$: in our language, they construct a map
\[
   BPU \heq K (\Z,3) \to B\GL{K}.
\]
From our point of view the existence of this map can be phrased as a
question about the $\Spinc$ orientation of complex $K$-theory.


To see this, recall that Atiyah, Bott, and Shapiro \cite{ABS:CM}
produce a Thom isomorphism  in complex $K$-theory for
$\Spinc$-bundles.  According to Theorem
\ref{t-th-inf-ii-or-thy-infty-lifting}, this corresponds to
the arrow labeled $ABS$ in the diagram
\[
\xymatrix{
{K (\Z,2)}
 \ar[rr]
 \ar[d]
& &
{\GL{K}}
 \ar[d]
\\
{B\Spinc}
 \ar[rr]^-{ABS}
 \ar[d]
& &
{\Triv{K}\heq \ptspace }
 \ar[d]
 \\
{BSO}
 \ar[r]^{BJ}
 \ar[d]_-{\beta w_{2}}
&
{\Line{S}}
 \ar[r]^-{(\slot)\Smash K}
 &
{\Line{K}\heq B\GL{K}}
 \ar@{=}[d]
 \\
{K (\Z,3)}
 \ar@{-->}[rr]^-{BABS}
& &
{B\GL{K}}
}
\]
The map $ABS$ induces a map
\begin{equation}\label{eq:33}
ABS\colon K (\Z,2)\to \GL{K},
\end{equation}
and the diagram
suggests that we ask whether this map deloops to give a map
\[
   B(ABS)\colon K (\Z,3) \to B\GL{K},
\]
as indicated.    Not surprisingly, this question is related to the
multiplicative properties of the orientation.

The Thom spectrum associated $M\Spinc$ associated to $B\Spinc$ is
a commutative $S$-algebra, as is $K$-theory.  The construction of
Atiyah-Bott-Shapiro produces a map of spectra
\[
   t\colon M\Spinc \to K,
\]
and Michael Joachim \cite{MR2122155} shows that $t$ can be refined to
a map of commutative $S$-algebras.  As we shall see in
\S\ref{sec:mult-orient-comp}, it follows that $K(\Z,2)\to \GL{K}$ is a
map of infinite loop spaces.

We use these ideas to twist $K$-theory by maps $X\to K (\Z,3)$ in
\S\ref{sec:application:-k-z}.

\section{Multiplicative orientations and comparison of Thom spectra}
\label{sec:mult-orient-comp}

The most familiar orientations are exponential: the Thom class of a
Whitney sum is the product of the Thom classes.  For example, consider
the case of Spin bundles and real $K$-theory, $KO$.  Atiyah, Bott, and
Shapiro show that the Dirac operator associates to a spin vector
bundle $V\to X$ a Thom class $t (V) \in KO (X^{V}).$  If $MSpin$ is
the Thom spectrum of the universal $Spin$ bundle over $BSpin$, then we
can view their construction as corresponding to a map of spectra
\[
t\colon MSpin \to KO.
\]

If $W\to Y$ is another spin vector bundle, then
\[
    (X\times Y)^{V\oplus W} \heq X^{V} \Smash Y^{W},
\]
and it turns out that with respect to the resulting isomorphism
\[
      KO (    (X\times Y)^{V\oplus W})\iso KO (X^{V} \Smash Y^{W}),
\]
one has
\begin{equation}\label{eq:34}
    t (V\oplus W) = t (V) \Smash t (W).
\end{equation}

Now the sum of vector bundles gives $MSpin$ the structure of a ring
spectrum, and so the multiplicative property \eqref{eq:34} (together
with a unit condition, which says that $t (\uln{0}) = 1$) corresponds
to the fact that
\[
   t\colon MSpin\to KO
\]
is a map of monoids in the homotopy category of spectra.

It is important that $t$ is in fact a map of commutative monoids in
the $\i$-category of spectra.  More precisely, $MSpin$ and $KO$ are both
commutative $S$-algebras, and it turns out
\cite{MR2122155,AHR:orientation} that $t$ is a map of commutative $S$-algebras.

The construction of classical Thom spectra such as $MSpin$, $MSO$,
$MU$ as commutative $S$-algebras (equivalently, $E_{\infty}$ ring
spectra) is due to \cite{MQRT:ersers,LMS:esht}.   In this section, we discuss
the theory from the $\i$-categorical point of view.   We'll see (Remark
\ref{rem-2}) that
this gives a way to think about the comparison of our Thom spectrum
to classical constructions.  It also provides some tools we use to
build twists of $K$-theory from $PU$-bundles.

We begin with a question.  Suppose that $A$ is a commutative
$S$-algebra.  Under what conditions on a map $f$ should we expect that the Thom
\[
         X^{f} = \colim (\Sing X \xra{f} \Line{A} \xra{j} \Mod{A})
\]
is a commutative $A$-algebra?  And in that situation, how do we
understand $A$-algebra maps out of $X^{f}$?

In the context of $\i$-categories, spaces play the role which sets
play in the context of classical categories, and so we begin by studying
the situation of a discrete commutative ring $R$ and a set $X$.
In that case, an $R$-line is just a free rank-one $R$-module, and a
bundle of $R$-lines  $\xi$ over $X$ is just a collection of $R$-lines,
indexed by the points $x\in X$.   We can think of this as a functor
\[
   \xi\colon X \to \Line{R}
\]
from $X$, considered as a discrete category, to the category of
$R$-lines: free rank-one $R$-modules and isomorphisms.  The ``Thom
spectrum''
\[
  X^{\xi} = \colim \left( X \xra{\xi} \Line{R} \xra{}\Mod{R} \right)
\]
is easily seen to be the sum
\[
    X^{\xi} \iso \bigoplus_{x\in X} \xi_{x}.
\]

Now suppose that $R$ is a commutative ring, so that $\Mod{R}$ is a
symmetric monoidal category, and $\Line{R}$ is the maximal
sub-groupoid of $\Mod{R}$ generated by $R$.  If $X$ is a
discrete abelian group, then we may consider $X$ as a symmetric
monoidal category with objects the elements of $X.$  It is then not
difficult to check the following.

\begin{Proposition} \label{t-pr-discrete-sym-mon-functor-com-R-alg}
If $\xi\colon X \to \Line{R}$ is a map of symmetric monoidal categories,
then $X^{\xi}$ has structure of a commutative $R$-algebra.
\end{Proposition}

The analogue of this result holds in the $\i$-categorical setting; see
Theorem \ref{t-pr-thom-spec-comm-A-alg} below.  It is possible to give
a direct proof; instead we sketch the circuitous proof given in
\cite{0810.4535v3}, as some of the results which arise along the way
will be useful in sections \ref{sec:application:-k-z} and
\ref{sec:appl-degr-four}.

We begin with another construction of the $R$-module $X^{\xi}$ in the
discrete associative case.  Let $\GL{R}$ be the group of units of $R$.
Note that the free abelian group functor
\begin{equation}\label{eq:35}
   \Z\colon \CatOf{sets} \to \CatOf{abelian goups}
\end{equation}
induces a functor
\begin{equation}\label{eq:36}
\xymatrix{
{\Z\colon \CatOf{groups}}
 \ar@<+1ex>[r]
&
{\CatOf{rings}\colon \GLsym}
 \ar@<+1ex>[l]
}
\end{equation}
whose right adjoint is $\GLsym$.  In particular, we have a natural map
of rings
\[
     \Z[\GL{R}] \to R
\]
and so the colimit-preserving functor
\[
    \CatOf{$\GL{R}$-sets} \xra{\Z[\slot]\otimes_{\Z[\GL{R}]}R} \Mod{R}.
\]
This functor restricts to an equivalence of categories
\begin{equation}\label{}
\xymatrix{
{\Z[\slot]\otimes_{\Z[\GL{R}]} R\colon \Tors{\GL{R}}}
 \ar@<+1ex>[r]
&
{\Line{R}\colon T}.
 \ar@<+1ex>[l]
}
\end{equation}
Here $\Tors{\GL{R}}$ is the category of $\GL{R}$-torsors, and the
inverse equivalence is the functor $T$ which associates to an
$R$-line $L$ the $\GL{R}$-torsor
\[
T (L) \eqdef \Line{R} (R,L) \iso  \{u \in L | Ru \iso L \}.
\]

That is, we have the following diagram of categories which commutes up
to natural isomorphism
\[
\xymatrix{
{\Line{R}}
 \ar[r]
 \ar@<+1ex>^{T}_{\iso}[d]
&
{\Mod{R}}
\\
{\Tors{\GL{R}}}
 \ar@<+1.5ex>[u]^{\Z[\slot]\otimes_{\Z[\GL{R}]} R}
 \ar[r]
&
{\CatOf{$\GL{R}$-sets}.}
 \ar[u]_{\Z[\slot]\otimes_{\Z[\GL{R}]} R}
}
\]
Moreover, the vertical arrows preserve colimits, and the left vertical
arrows comprise an equivalence.

If $\xi$ is a bundle of $R$-lines over $X$, we write $P (\xi)$ for the
$\GL{R}$-set
\[
   P (\xi) = \colim\left(X \xra{\xi} \Line{R}\xra{T}
   \Tors{\GL{R}}
 \xra{} \CatOf{$\GL{R}$-sets}\right).
\]
That is, $P (\xi)$ is the $\GL{R}$-torsor over $X$ whose fiber at
$x\in X$ is  $P (\xi)_{x} = T (\xi_{x}).$

\begin{Proposition} \label{t-pr-comparison-disc-thom-sp}
There is a natural isomorphism of $R$-modules
\[
   X^{\xi} \iso \Z[P (\xi)]\otimes_{\Z[\GL{R}]} R.
\]
\end{Proposition}

\begin{proof}
We have
\begin{align*}
     X^{\xi} &= \colim\left(X \xra{\xi} \Line{R} \to \Mod{R}\right) \\
             &\iso
\colim\left( X \xra{T\xi} \Tors{\GL{R}} \xra{} \CatOf{$\GL{R}$-sets}
 \xra{\Z[\slot]\otimes_{\Z[\GL{R}]}R} \Mod{R} \right) \\
& \iso \Z[P (\xi)]\otimes_{\Z[\GL{R}]} R.
\end{align*}
\end{proof}

Now suppose that $R$ is a commutative ring, and $X$ is an abelian
group.  If $\xi\colon X\to \Line{R}$ is a symmetric monoidal functor, then
\[
     \GL{R} \to P (\xi) \to X
\]
is an extension of abelian groups.  The adjunction
\eqref{eq:36} restricts further to an adjunction
\begin{equation}\label{eq:38}
\xymatrix{
{\Z\colon \CatOf{abelian groups}}
 \ar@<+1ex>[r]
&
{\CatOf{commutative rings}\colon \GLsym,}
 \ar@<+1ex>[l]
}
\end{equation}
and so we have maps of commutative rings
\[
      \Z[P (\xi)] \leftarrow \Z[\GL{R}] \rightarrow R.
\]
The isomorphism of Proposition \ref{t-pr-comparison-disc-thom-sp} has
the following consequence.

\begin{Proposition}\label{t-pr-thom-spec-comm-disc}
If $R$ is a commutative ring and $\xi\colon X\to \Line{R}$ is a
symmetric monoidal functor, then $X^{\xi}$ is a commutative $R$-algebra;
indeed, it is the pushout in the category of commutative rings
\[
\begin{CD}
     \Z[\GL{R}] @>>> R \\
     @VVV @VVV \\
     \Z[P (\xi)] @>>> X^{\xi}.
\end{CD}
\]
\end{Proposition}

The preceding discussion generalizes elegantly and directly to spaces
and spectra.  The generalization illustrates  how spaces
play the role in $\i$-categories that sets play in categories.  The
reader may find it useful to consult the table in
Figure~\ref{fig-class-i-cat}.

\begin{figure}\label{fig-class-i-cat}
\caption{Selected instances of the analogy
  Sets:Spaces::Categories:$\i$-categories}
\vspace{5 pt}
\begin{tabular}{|l|l|l|}
\hline
Categorical notion & $\i$-categorical notion  & Alternate  name/description  \\
\hline
\hline
Category & $\i$-category & \\
\hline
Set & $\i$-groupoid & Space/Kan complex  \\
\hline
Monoid & monoidal $\i$-groupoid & $A_{\infty}$ space; $\ptspace$-algebra\\
\hline
Group    & group-like monoidal   & group-like $\ainfty$ space\\
& $\i$-groupoid & \\
\hline
Abelian group & group-like symmetric  & group-like $E_{\infty}$ space; \\
              & monoidal $\i$-groupoid & $(-1)$-connected spectrum  \\
\hline
Abelian group & spectrum &  \\
\hline
The ring $\Z$ & The sphere spectrum $S$ & \\
\hline
Ring & Monoid in spectra & $S$-algebra or $A_{\infty}$ ring spectrum\\
\hline
Commutative ring & Commutative monoid in spectra & Commutative
$S$-algebra or $E_{\infty}$ ring spectrum \\
\hline
The functor $\Z$ & The functor $\splus$ &  \\
\hline
$\Mod{R}$      & $\Mod{A}$  & \\
\hline
$\Line{R}$     & $\Line{A}$ & $B\GL{A}$\\
\hline
$\GL{R}$       & $\Line{A} (A,A) =  \Aut_{A} (A)$ & $\GL{A}$\\
\hline
$\GL{R}$-set &  $\GL{A}$-space  & $A_{\infty}$ $\GL{A}$-space  \\
\hline
$\GL{R}$-torsors & $\Tors{\GL{A}}$ &  \\
\hline
$\otimes$ & $\Smash$ &
\\
\hline
\end{tabular}
\end{figure}

Let $\spaces$ be an $\i$-category of spaces.  The
analogue of the adjunction \eqref{eq:35} is
\begin{equation}\label{eq:39}
\xymatrix{
{\splus\colon \spaces}
 \ar@<+1ex>[r]
&
{\spectra\colon \linf.}
 \ar@<+1ex>[l]
}
\end{equation}

\begin{Definition}
Let $\Alg (\ptspace)$ be the $\i$-category of monoids in the
$\i$-category of $\spaces$.
\end{Definition}

We introduce the awkward name $\ptspace$-algebra to emphasize that
these are more general than monoids with respect to the classical
product of topological spaces.  The symmetric monoidal structure on
$\spaces$ is such that $\Alg (\ptspace)$ is a model for the
$\i$-category of $\ainfty$ spaces.  A
$\ptspace$-algebra $X$ is \emph{group-like} if $\pi_{0}X$ is a group.

\begin{Definition}
We write $\Alg (\ptspace)^{\times}$ for the $\i$-category of
group-like monoids.
\end{Definition}

The adjunction \eqref{eq:39} restricts to an adjunction
\begin{equation} \label{eq:40}
\xymatrix{
{\splus\colon \Alg (\ptspace)^{\times}}
 \ar@<+1ex>[r]
&
{\Alg (S)\colon \GLsym.}
 \ar@<+1ex>[l]
}
\end{equation}

\begin{Remark}
If $X$ is a group-like monoid in $\spaces$,
then $\Sing X$ is a group-like monoidal $\i$-groupoid.  One way to see
that a group-like $\ainfty$ space is the appropriate generalization of
a group is to observe that if $Z$ is an object in a  \emph{category},
then $\End (Z)$ is a monoid, and $\Aut (Z)$ is a
group.  If
$Z$ is an object in an \emph{$\i$-category}, then $\End
(Z)$ is a monoidal $\i$-groupoid, and $\Aut (Z)$ is group-like.
\end{Remark}

In particular, as we have already discussed in
\S\ref{sec:line_a-gla}, one construction of the right adjoint
$\GLsym$ is the following.   If $A$ is an $S$-algebra, then it has an
$\i$-category of modules $\Mod{A}$.  In $\Mod{A}$ we have the maximal
sub-$\i$-groupoid $\Line{A}$ whose objects are weakly equivalent to
$A$.  We can define
\[
   \GL{A} = \Aut (A) = \Line{A} (A,A)
\]
to be the subspace of (the geometric realization of) $\Mod{A} (A,A)$
consisting of homotopy equivalences: it is a group-like
$\ptspace$-algebra.    We also write
$B\GL{A}$ for the full $\i$-subcategory of $\Line{A}$ on the single object
$A$: this is the $\i$-groupoid with $\Aut{A}$ as its simplicial set of
of morphisms.

Since $\GL{A}$ is a (group-like) monoid in the
symmetric monoidal
$\i$-category of spaces, we can form the $\i$-category
of $\GL{A}$-spaces,  and then define $\Tors{\GL{A}}$ to be the maximal subgroupoid
whose objects are $\GL{A}$-spaces $P$ which admit an equivalence of
$\GL{A}$-spaces
\[
    \GL{A} \heq P.
\]

The adjunction \eqref{eq:40} provides
a map of $S$-algebras
\[
   \splus \GL{A} \to A,
\]
and so a ($\i$-category) colimit-preserving functor
\[
   \splus (\slot)\Smash_{\splus \GL{A}}A\colon \CatOf{$\GL{A}$-spaces} \xra{}
   \Mod{A},
\]
which restricts to an equivalence of $\i$-categories
\[
 \splus (\slot)\Smash_{\splus \GL{A}}A \colon \Tors{\GL{A}} \to \Line{A}.
\]
The inverse equivalence is the functor
\[
   T\colon \Line{A} \to \Tors{\GL{A}}
\]
which to an $A$-line $L$ associates the $\GL{A}$-torsor
\[
T (L) \eqdef \Line{A} (A,L).
\]

Putting all these together, we have the homotopy commutative diagram
of $\i$-categories
\[
\xymatrix{
{\Line{A}}
 \ar[r]
 \ar@<+1ex>^{T}_{\heq}[d]
&
{\Mod{A}}
\\
{\Tors{\GL{A}}}
 \ar@<+1.5ex>[u]^{\splus (\slot)\Smash_{\splus \GL{A}} A}
 \ar[r]
&
{\CatOf{$\GL{A}$-spaces},}
 \ar[u]_{\splus (\slot)\Smash_{\splus \GL{A}} A}
}
\]
in which the vertical arrows preserve $\i$-categorical colimits, and
the left vertical arrows comprise an equivalence.

Now let $X$ be a space, and let
\[
  \xi\colon X \to \Line{A}
\]
be a bundle of $A$-lines over $X$.  Recall that
\[
   X^{\xi} = \colim \left(\Sing X \xra{\xi} \Line{A} \xra{} \Mod{A}\right).
\]
On the other hand, let
\[
   P (\xi) = \colim \left(\Sing X \xra{\xi} \Line{A} \xra{T} \Tors{\GL{A}}
   \xra{} \CatOf{$\GL{A}$-spaces}\right).
\]
We have the following analogue of Proposition
\ref{t-pr-comparison-disc-thom-sp}.

\begin{Proposition} \label{t-pr-thom-spectrum-tensor-prod}
There is a natural equivalence of $A$-modules
\begin{equation}\label{eq:9}
   X^{\xi} \heq \splus P (\xi)\Smash_{\splus \GL{A}} A.
\end{equation}
\end{Proposition}

Now we turn to the commutative case.

\begin{Definition}
We write $\CommAlg (\ptspace)$ for the $\i$-category of commutative
monoids in $\spaces$.  It is equivalent to the nerve of the simplicial category of (cofibrant and fibrant) $\einfty$
spaces.  We write $\CommAlg (\ptspace)^{\times}$ for the $\i$-category
of group-like commutative monoids, which models group-like $\einfty$
spaces.
\end{Definition}

The adjunction \eqref{eq:40} restricts to the analogue of the
adjunction \eqref{eq:38}, namely
\begin{equation} \label{eq:41}
\xymatrix{
{\splus\colon \CommAlg (\ptspace)^{\times}}
 \ar@<+1ex>[r]
&
{\CommAlg (S)\colon \GLsym}
 \ar@<+1ex>[l]
}
\end{equation}

The reader will notice that in the table in Figure
\ref{fig-class-i-cat}, we mention  two models
for ``abelian groups'', namely, group-like $\einfty$ spaces and
spectra.  It
is a classical theorem of May \cite{May:gils,MR0339152}, reviewed for
example in \cite[\S3]{0810.4535v3},
that the functor $\linf$ induces an equivalence
of $\i$-categories
\[
    \linf\colon \CatOf{$(-1)$-connected spectra} \heq \CommAlg (\ptspace)^{\times},
\]
and so we may rewrite the adjunction \eqref{eq:41} as
\begin{equation}\label{eq:46}
\xymatrix{
{\splus\linf\colon \CatOf{$(-1)$-connected spectra}}
 \ar@<-1ex>[r]_-{\linf}^-{\heq}
&
{\CommAlg (\ptspace)^{\times}}
 \ar@<+1ex>[r]^-{\splus}
 \ar@<-1ex>[l]
&
{\CommAlg (S)\colon \glsym}
 \ar@<+1ex>[l]^-{\GLsym}
}
\end{equation}
(The left adjoints are written on top, but the pair of adjoints
on the left is an equivalence of $\i$-categories).  Note that we have
introduced the functor
$\glsym,$ with the property that if $A$ is a commutative $S$-algebra,
then
\[
       \GL{A} \heq \linf \gl{A}.
\]
We also define $b\gl{A} = \Sigma \gl{A}$: then $\linf b\gl{A} \heq B\GL{A}$.

Thus a map of spectra
\[
    f\colon b\to b\gl{A}
\]
may be viewed equivalently as a map of group-like commutative
$\ptspace$-algebras
\[
    \linf f\colon B = \linf b \to B\GL{A}
\]
or as a map of symmetric monoidal $\i$-groupoids
\begin{equation}\label{eq:8}
   \xi\colon    \Sing B \xra{} \Line{A}.
\end{equation}

Form the pull-back diagram
\begin{equation}\label{eq:43}
\begin{CD}
\gl{A} @= \gl{A} \\
@VVV    @VVV \\
p  @>>> e\gl{A}\heq\ptspace \\
@VVV @VVV \\
b @> f >> b\gl{A}.
\end{CD}
\end{equation}
One checks \cite[Lemma 8.23]{0810.4535v3} that
\[
      P (\xi) \heq \linf p,
\]
and so we have the following.

\begin{Proposition}  \label{t-pr-thom-spec-comm-A-alg}
The Thom spectrum of the map $\xi$ of symmetric monoidal
$\i$-groupoids \eqref{eq:8}
is a commutative $A$-algebra; indeed, we have
\[
     X^{\xi} \heq \splus P (\xi) \Smash_{\splus \GL{A}} A \heq
       \splus \linf p \Smash_{\splus \linf \gl{A}} A.
\]
\end{Proposition}

\begin{Example}
Taking $A$ to be the sphere spectrum in Proposition
\ref{t-pr-thom-spec-comm-A-alg}, we recover the result of
\cite{LMS:esht} that the Thom spectrum of an $\i$-loop map
\[
   B \to B\GL{S}
\]
is a commutative $S$-algebra.
\end{Example}

\begin{Remark} \label{rem-2}
The formula \eqref{eq:9} provides one way to see that our Thom
spectrum coincides with the classical Thom spectrum of
\cite{MQRT:ersers,LMS:esht}.  One way to compute the smash product in
\eqref{eq:9} is to realize it as a two-sided bar
construction \cite[Proposition 7.5]{EKMM}.  In particular if $\xi\colon
X\to B\GL{S}$ classifies a
spherical fibration, then one expects
\begin{equation}\label{eq:10}
 X^{\xi} \heq \splus P (\xi)\Smash_{\splus \GL{S}} S \heq B (\splus P
 (\xi), \splus \GL{S}, S).
\end{equation}
It is not difficult to see that the constructions of
\cite{MQRT:ersers,LMS:esht} provide careful models for the two-sided
bar construction on the right-hand side of \eqref{eq:10}.
Note that some care is required to make this proposal precise; see
\S8.6 of \cite{0810.4535v3} for details.
\end{Remark}

\section{Application: $K (\Z,3),$ twisted $K$-theory, and the
$Spin^{c}$ orientation}
\label{sec:application:-k-z}

\subsection{}
\label{sec:formal-Kthy-twists}

Recall that $BSpin^{c}$ participates in a fibration of infinite loop
spaces
\begin{equation}\label{eq:42}
   K (\Z,2) \rightarrow \Bspinc  \rightarrow BSO \xra{bw_{2}} K (\Z,3),
\end{equation}
where $bw_{2}$ is the composite of the usual $w_{2}$ with
the $\Z$-Bockstein
\[
   BSO \xra{w_{2}} K (\Z/2,2) \xra{b} K (\Z,3).
\]
Passing to Thom spectra in \eqref{eq:42} we have a map of commutative
$S$-algebras
\[
   \splus K (\Z,2) \to MSpin^{c}.
\]
It's a theorem of Joachim  \cite{MR2122155} that the orientation
\[
   MSpin^{c} \to K
\]
of Atiyah-Bott-Shapiro is map of commutative $S$-algebras, and so we
have a sequence of maps
of commutative $S$-algebras
\begin{equation} \label{eq:49}
    \splus K (\Z,2) \rightarrow MSpin^{c} \to K.
\end{equation}
The $(\splus\linf,\glsym$) adjunction \eqref{eq:46}
produces from \eqref{eq:49} a map of infinite loop spaces
\[
   K (\Z,2) \rightarrow \GL{K}
\]
which we deloop once to view as a map
\[
   T\colon K (\Z,3) \rightarrow B\GL{K}.
\]
That is, the fact that the Atiyah-Bott-Shapiro orientation is a map of
(commutative) $S$-algebras implies that we have a homotopy-commutative
diagram
\begin{equation}\label{eq:48}
\begin{CD}
K (\Z,2) @>>> \GL{K} \\
@VVV @VVV \\
\Bspinc @>>> \ptspace \\
@VVV @VVV \\
BSO @> j >> B\GL{K} \\
@V \beta w_{2} VV @VV = V \\
K (\Z,3) @> T >> B\GL{K}.
\end{CD}
\end{equation}

Now suppose given a map $\alpha\colon X \to K (\Z,3)$.  Then we may form
the Thom spectrum $X^{T\alpha}$, and define
\[
   K^{n} (X)_{\alpha} \eqdef \pi_{0}\Mod{K} (X^{T\alpha}, \Sigma^{n}K).
\]
We then have the following.

\begin{Proposition}
A map $\alpha\colon X \to K (\Z,3)$ gives rise to a twist $K^{*} (X)_{\alpha}
$ of the $K$-theory of $X$.  A choice of homotopy
\[
    T \beta w_{2} \Rightarrow j
\]
in the diagram \eqref{eq:48} above determines, for every oriented
vector bundle $V$
over $X$, an isomorphism
\begin{equation}\label{eq:47}
    K^{*} (X)_{\beta w_{2} (V)} \iso K^{*} (X^{V}).
\end{equation}
In this way, the characteristic class $\beta w_{2} (V)$ determines the
$K$-theory of the Thom spectrum of $V$.
\end{Proposition}

\begin{proof}
We prove the isomorphism \eqref{eq:47} to show how simple it is from
this point of view.  Consider the homotopy commutative diagram
\[
\begin{CD}
X @> V >> BSO @> j >> B\GL{K} \\
@. @V \beta w_{2} VV @VV = V \\
@. K (\Z,3) @> T >> B\GL{K}.
\end{CD}
\]
Omitting the gradings, we have
\begin{align*}
   K^{0} (X)_{\beta w_{2} (V)} & = \pi_{0}\Mod{K} (X^{T\beta w_{2} (V)},K) \\
 & \iso \pi_{0}\Mod{K} (X^{j (V)},K) \\
 & \iso \pi_{0}\Mod{K} (X^{V}\Smash K,K) \\
 & \iso \pi_{0} \spectra (X^{V},K) \\
 & = K^{0} (X^{V}).
\end{align*}
The first isomorphism uses the construction of $X^{\xi}$ together with
the fact that $T\beta w_{2}$ and $j$ are homotopic as maps $BSO\to
B\GL{K}$.  The second isomorphism is Theorem
\ref{t-th-i-cat-thom-is-thom}.
\end{proof}

\begin{Remark}  \label{rem-1}
We needn't have started with an oriented bundle.  For example, let $F$
be the fiber in the sequence
\[
   F \to BSpin^{c} \to BO.
\]
This is a fibration of infinite loop spaces, and so
it deloops to give
\[
   F \to BSpin^{c} \to BO \xra{\gamma} BF.
\]
The same argument produces an $\einfty$ map
\[
   \splus F \to MSpin^{c} \to K
\]
whose adjoint
\[
    F\to \GL{K}
\]
deloops to
\[
    \zeta\colon BF \to B\GL{K},
\]
and if $V$ is any vector bundle then
\[
   K (X^{V}) \iso K (X)_{\zeta \gamma (V)}.
\]
\end{Remark}

\subsection{Khorami's theorem}

At this point we are in a position to state a remarkable result of
M. Khorami \cite{Khorami}.  $K (\Z,2)$ is a group-like commutative
$\ptspace$-algebra, and so we have a bundle
\[
      K (\Z,2) \to EK (\Z,2) \heq \ptspace \to BK (\Z,2) \heq K (\Z,3).
\]
(One can build this bundle a number of ways: by modeling $K (\Z,2)\heq
PU$ as mentioned in \S\ref{sec:geom-models-twist},  or using infinite
loop space theory, or using the $\i$-category technology described above.)

Given a map $\zeta\colon X\to K (\Z,3)$, we can form the pull-back $K (\Z,2)$-bundle
\[
\begin{CD}
Q @>>> EK (\Z,2) \\
@VVV @VVV \\
X @>\zeta>> K (\Z,3).
\end{CD}
\]
On the other hand, let $P$ be the pull-back
\[
\begin{CD}
P @>>> E\GL{K} \\
@VVV @VVV \\
X @> T\zeta >> B\GL{K}
\end{CD}
\]
as in Proposition \ref{t-pr-thom-spec-comm-A-alg}.  One can check that
\[
     \splus P  \heq \splus Q \Smash_{\splus K (\Z,2)} \splus \GL{K},
\]
and so the Thom spectrum whose homotopy calculates the $\zeta$-twisted
$K$-theory is
\begin{equation}\label{eq:18}
    X^{T\zeta} \heq  \splus P \Smash_{\splus \GL{K}} K \heq
    \splus Q \Smash_{\splus K (\Z,2)} \splus \GL{K} \Smash_{\splus
    \GL{K}} K \heq  \splus Q \Smash_{\splus K (\Z,2)} K,
\end{equation}
where the map of commutative $S$-algebras $\splus K (\Z,2) \to K$ is
\eqref{eq:49}.

The formula \eqref{eq:18} implies that there is a spectral
sequence (see for example \cite[Theorem 4.1]{EKMM})
\begin{equation}\label{eq:20}
    \Tor^{K_{*} K (\Z,2)}_{*} (K_{*}Q, K_{*}) \Rightarrow
    \pi_{*}X^{T\zeta} \iso K_{*} (X)_{\zeta}.
\end{equation}
Note that $K_{*} K (\Z,2)\iso K_{*}\{\beta_{1},\beta_{2},\ldots \},$ so $K_{*}$
is not a flat $K_{*}K (\Z,2)$-module.  Nevertheless Khorami proves the
following.

\begin{Theorem}
In \eqref{eq:20} one has $\Tor_{q}=0$ for $q>0$, and so
\[
  K_{*} (X)_{\zeta} \iso K_{*} Q \otimes_{K_{*} K (\Z,2)}  K_{*}.
\]
\end{Theorem}

\section{Application: Degree-four cohomology and twisted elliptic cohomology}
\label{sec:appl-degr-four}

The arguments of \S\ref{sec:formal-Kthy-twists} apply equally well to
$String$ structures and the spectrum of topological modular forms.

Recall that spin vector bundles admit a degree four characteristic
class $\lambda$, which we may view as a map of infinite loop spaces
\[
    BSpin \xra{\lambda} K (\Z,4).
\]
Indeed this map detects the generator of $H^{4}BSpin \iso \Z.$
The fiber of $\lambda$ is called $\Bstring,$ and so we have maps of
infinite loop spaces
\[
   K (\Z,3) \xra{} \Bstring \to BSpin \xra{\lambda} K (\Z,4).
\]
Passing to Thom spectra, we get maps of commutative $S$-algebras
\[
    \splus K (\Z,3) \to \mstring \to MSpin.
\]
Let $\TMF$ be the spectrum of topological modular forms
\cite{Hopkins:icm2002}.   Ando, Hopkins, and Rezk
\cite{AHR:orientation} have produced a map of commutative $S$-algebras
\[
    \mstring \xra{\sigma} \TMF,
\]
and so we have a map of commutative $S$-algebras
\[
   \splus K (\Z,3) \to \TMF,
\]
whose adjoint (see (\ref{eq:40},\ref{eq:41},\ref{eq:46}))
\[
    K (\Z,3) \to \GL{\TMF}
\]
deloops to
\[
    T\colon K (\Z,4) \to B\GL{\TMF}.
\]
By construction, the map $T$ makes the diagram
\begin{equation}\label{eq:50}
\begin{CD}
K (\Z,3) @>>> \GL{\TMF} \\
@VVV @VVV \\
\Bstring @>>> \ptspace \\
@VVV @VVV \\
BSpin @> j >> B\GL{\TMF} \\
@V \lambda VV @VV = V \\
K (\Z,4) @> T >> B\GL{\TMF}
\end{CD}
\end{equation}
commute up to homotopy.  If $\zeta\colon X \to K (\Z,4)$ is a map, then we
may define
\[
    \TMF (X)^{k}_{\zeta} \eqdef \pi_{0}\Mod{\TMF} (X^{T\zeta},\Sigma^{k}\TMF),
\]
and so we have the following.

\begin{Proposition}
A map $\zeta\colon X \to K (\Z,4)$ gives rise to a twist $\TMF^{*} (X)_{\zeta}
$ of the $\TMF$-theory of $X$.  A choice of homotopy
\[
    T \lambda \Rightarrow j
\]
in the diagram \eqref{eq:50} above determines, for every map
\[
   V\colon X \xra{} BSpin,
\]
an isomorphism of $\TMF^{*} (X)$-modules
\begin{equation}\label{eq:7}
    \TMF^{*} (X)_{\lambda  (V)} \iso \TMF^{*} (X^{V}).
\end{equation}
In this way, the characteristic class $\lambda (V)$ determines the
$\TMF$-cohomology of the Thom spectrum of $V$.
\end{Proposition}

\begin{Remark}\label{rem-3}
As in Remark \ref{rem-1}, we could have started with the fiber $F$
in the fibration sequence of infinite loop spaces
\[
   F \to \Bstring  \to BO \xra{\gamma} BF.
\]
We have a map of commutative $S$-algebras
\[
   \splus F \to \mstring \to \TMF
\]
whose adjoint
\[
    F\to \GL{\TMF}
\]
deloops to
\[
    \zeta\colon BF \to B\GL{\TMF},
\]
and if $V\colon X\to BO$ classifies a virtual vector bundle, then
\[
   \TMF^{*} (X^{V}) \iso \TMF^{*} (X)_{\zeta\gamma (V)}
\]
\end{Remark}

\section{Application: Poincare duality and twisted umkehr maps}
\label{sec:appl-poinc-dual}

Let $M$ be a compact smooth manifold with tangent bundle $T$ of rank
$d$.   Embed $M$ in $\R^{N}$, and then perform the Pontrjagin-Thom
construction: collapse to a point the complement of a tubular
neighborhood of $M$.  If $\nu$ is the normal bundle of the embedding,
this gives a map
\[
       S^{N} \to M^{\nu}.
\]
Desuspending $N$ times then yields a map
\begin{equation}\label{eq:2}
    \mu\colon S^{0} \to M^{-T}.
\end{equation}
As usual, the Thom spectrum admits a relative diagonal map
\[
     M^{-T}\xra{\Delta} \splus M \Smash M^{-T}
\]
(this is the map which gives the cohomology of the Thom spectrum the
structure of a module over the cohomology of the base).

If $E$ is a spectrum, then to a map
\[
    f\colon M^{-T}\to E
\]
we can associate the composition
\[
   S^{0}\xra{\mu} M^{-T} \xra{\Delta} \splus M \Smash M^{-T} \xra{
   1\Smash f} \xra{} \splus M \Smash E.
\]
Milnor-Spanier-Atiyah duality says that this procedure yields an isomorphism

\[
     E_{*} (M) \iso E^{-*} (M^{-T}).
\]
In the presence of a Thom isomorphism
\[
    E^{-*} (M^{-T}) \iso E^{d-*} (M)
\]
we have Poincar\'e duality
\[
    E_{*} (M) \iso E^{d-*} (M).
\]

Without a Thom isomorphism, we choose a map $\alpha$
\[
    M \xra{\alpha} BO
\]
classifying $\uln{d}-T$, and then we define $\tau (-T)$ to be the composition
\[
  \tau (-T)\colon  M \xra{\alpha} BO \xra{j} B\GL{S} \xra{} B\GL{E}.
\]
Then, following Example \ref{ex-9}, we have
\[
    E^{d-*} (M)_{\tau (-T)} \iso E^{d-*} (M^{\uln{d}-T}) \iso E^{-*}
    (M^{-T}) \iso E_{*} (M).
\]

Combining this with the results of sections
\ref{sec:formal-Kthy-twists} and \ref{sec:appl-degr-four}, we have the
following.

\begin{Proposition}
Suppose that $M$ is an oriented compact manifold of dimension $d$.  Then
\[
    K_{*} (M) \iso K^{d-*} (M)_{-\beta w_{2} (M)}.
\]
If $M$ is a spin manifold, then
\[
    \TMF_{*} (M)\iso \TMF^{d-*} (M)_{-\lambda (M)}.
\]
\end{Proposition}

\subsection{Twisted umkehr maps}
\label{sec:twisted-umkehr-maps}

In this section we sketch the construction of some umkehr maps in
twisted generalized cohomology.  Note that
similar constructions are studied in
\cite{arixv:math/0507414,MR2271789,arXiv:math/0611225}.  Also, since this
paper was written, we have learned that Bunke, Schneider, and
Spitzweck have independently developed a similar approach to twisted
umkehr maps.

Suppose that we have a family of compact spaces over $X$, that is, a
map of $\i$-categories
\[
    \zeta\colon \Sing X \to \CatOf{compact spaces}.
\]
We also have the trivial map
\[
     \ptspace_{X}\colon \Sing X \xra{\ptspace} \CatOf{compact spaces}.
\]
If $M$ is a compact space, let
\[
   DM = F (\splus M,S^{0})
\]
be the Spanier-Whitehead dual of $M_{\pt}$: this is a functor of $\i$-categories
\[
    D\colon \CatOf{compact spaces}^{\text{op}} \to \spectra.
\]
The
projection
\[
    M \to \ptspace
\]
gives rise to a map of spectra
\begin{equation}\label{eq:1}
    S^{0}\iso D\ptspace \to DM.
\end{equation}
Indeed if $M$ is a compact manifold with tangent bundle $T$, then
Milnor-Spanier-Atiyah duality says that $DM\heq M^{-T}$, in such a way
that the Pontrjagin-Thom map \eqref{eq:2} identifies with \eqref{eq:1}.

In any case, let $S_{X}^{0} = D\ptspace_{X}$.  We have a natural map
\[
   u\colon S_{X}^{0} \to D\zeta
\]
of bundles of spectra over $X$.  Essentially, we are applying the map
\eqref{eq:1} fiberwise.

It follows from Proposition 7.7 of \cite{0810.4535v3} that
\[
      X^{S^{0}_{X}} \heq \splus X.
\]
As for $X^{D\zeta}$, in a forthcoming paper we prove the following.

\begin{Proposition}\label{t-pr-tangents-along-fiber}
Suppose that $\zeta$ arises from a bundle
\[
    Y \xra{f} X
\]
of compact manifolds, and let $Tf$ be its
bundle of tangents along the fiber.  Then
\[
    X^{D\zeta} \heq Y^{-Tf}.
\]
In particular, passing to Thom spectra on $u$ gives a map of spectra
\begin{equation} \label{eq:3}
   t\colon \splus X \to Y^{-Tf}.
\end{equation}
This map is equivalent to the classical stable transfer map associated
to $f$.
\end{Proposition}
The map $t$, and indeed the idea that it arises from applying the map
\eqref{eq:2} fiberwise,  is classical; see for example
\cite{MR0377873}.   Casting it in our setting enables us to construct
twisted versions.

More precisely, suppose that $R$ is a commutative  $S$-algebra, and
suppose given a bundle of $R$-lines over $X$
\[
   \xi\colon \Sing X \xra{} \Line{R}.
\]
We then have a map of bundles of $R$-lines over $X$
\[
  u\Smash \xi\colon \xi \heq S^{0}_{X} \Smash \xi\to D\zeta \Smash \xi.
\]
Thus we have constructed a twisted umkehr map
\[
      R^{*} (X^{D\zeta})_{\xi} \to R^{*} (X)_{\xi}.
\]
In the situation of Proposition \ref{t-pr-tangents-along-fiber}, we
have a twisted transfer map
\begin{equation}\label{eq:4}
      R^{*} (Y^{-Tf})_{\xi} \to R^{*} (X)_{\xi}.
\end{equation}
About this we show the following.

\begin{Proposition}   \label{t-pr-twisted-umkehr}
Suppose that
\[
Y \xra{Tf} BO \xra{}B\GL{S} \xra{}B\GL{R},
\]
regarded as a map $\Sing Y \to \Line{R}$,
is homotopic to $\xi f\colon \Sing Y \xra{}\Sing X \to \Line{R}.$  A choice
of homotopy determines an isomorphism
\[
    R^{*} (Y) \iso   R^{*} (Y^{-Tf})_{\xi},
\]
and composing with the twisted transfer \eqref{eq:4} we have a twisted
umkehr
map $R^{*} (Y) \to R^{*} (X)_{\xi}$.
\end{Proposition}

\section{Motivation: D-brane charges in $K$-theory}

\label{sec:motivation-d-brane-umkehr}

\subsection{The Freed-Witten anomaly}
\label{sec:freed-witten-anomaly}

Let $j\colon D \to  X$ be an embedded submanifold, let $\nu$ be the normal
bundle of $j$, and suppose that $D$ carries a
complex vector bundle $\xi.$

Suppose moreover that $\nu$ carries a $\Spinc$-structure.  Then we can
form the $K$-theory push-forward
\[
j_{!}\colon K (D) \to K (X).
\]
In that situation Minasian and Moore and Witten discovered that it is
sensible to think of the $K$-theory class
\[
j_{!} (\xi) \in K (X)
\]
as the ``charge'' of the $D$-brane $D$ with Chan-Paton bundle
$\xi.$

If $\nu$ does not carry a $\Spinc$-structure, then we still have the
Pontrjagin-Thom construction
\[
    X \to D^{\nu}.
\]
Suppose we have a map $H\colon X \to K (\Z,3)$ making the diagram
\[
\begin{CD}
    D @> \nu >> BSO \\
@V j VV @VV bw_{2} V \\
X @> H >> K (\Z,3)
\end{CD}
\]
commute up to homotopy.  According to Proposition
\ref{t-pr-twisted-umkehr}, a homotopy
\[
   c\colon bw_{2}\heq Hj
\]
determines an isomorphism
\[
     K^{*} (D) \iso  K^{*} (D^{\nu})_{-H}
\]
(since $\nu=-Tj$), and then we have a twisted umkehr map
\begin{equation}\label{eq:45}
   j_{!}\colon K^{*} (D) \xra{} K^{*} (X)_{-H}.
\end{equation}

The class $j_{!} (\xi) \in K^{*}_{-H} (X)$ is evidently an analogue of
the charge in this situation.  The discovery of the condition that
there exists a class $H$ on $X$ such that $H|_{D} = W_{3} (\nu)$ is
due to Freed and Witten \cite{fw:asd}.

Although we discovered this push-forward in an attempt to understand
Freed and Witten's condition, we were not the first: it appeared,
formulated this way, in a paper of Carey and Wang
\cite{arixv:math/0507414}.   An important contribution of their work
is the construction, using the twisted $K$-theory of \cite{MR2172633},
of the umkehr map \eqref{eq:45}.

\section{An elliptic cohomology analogue}
\label{sec:an-ellipt-cohom}

Now suppose that we are given an
embedding of manifolds
\[
    j\colon M\to Y,
\]
and that $\nu = \nu (j)$ is equipped with a Spin structure.  Suppose
we have a map $H$
making the diagram
\begin{equation}\label{eq:51}
\begin{CD}
   M @> \nu >> BSpin \\
@V j VV @VV \lambda V \\
  Y @>  H >> K (\Z,4)
\end{CD}
\end{equation}
commute up to homotopy.  By Proposition \ref{t-pr-twisted-umkehr}, a
homotopy $c\colon \lambda\nu \heq Hj$ determines an
isomorphism
\[
   \TMF^{*} (M)\heq \TMF^{*} (D^{\nu})_{-H},
\]
and then we have
a homomorphism of
$\TMF^{*} (Y)$-modules
\begin{equation} \label{eq:6}
  \TMF^{*} (M) \heq \TMF^{*} (D^{\nu})_{-H} \rightarrow \TMF^{*} (Y)_{-H}.
\end{equation}

\begin{Remark}
The data of a configuration like \eqref{eq:5} together with the homotopy $c$
was studied by Wang \cite{MR2438341}, who calls it a \emph{twisted $String$
structure}.   In fact, it was predicted by
Kriz and Sati \cite{MR2086925,Sati-TCU} that $\TMF$ should be the natural
receptacle for $M$-brane charges.
\end{Remark}

\begin{Remark}
The authors are grateful to Hisham Sati for suggesting that we think
about diagrams like \eqref{eq:51}.  The on-going investigation of the
resulting twisted umkehr maps is joint work with him.
\end{Remark}

\section{Twists of equivariant elliptic cohomology}
\label{sec:twists-equiv-ellipt}

At present we do not know how to twist equivariant cohomology theories
in general; for that matter, equivariant Thom spectra are poorly
understood.  However, twists by degree four Borel cohomology play an
important role in equivariant elliptic cohomology.  We review two
instances to give the reader a taste of the subject.

Let $G$ be a connected and compact Lie group.  In 1994, Grojnowski
sketched the construction of a $G$-equivariant elliptic cohomology
$E_{G}$, based on a complex elliptic curve of the form $C_{q} =
\C/\Lambda \iso
\C^{\times}/q^{\Z};$ more generally, the
construction can be used to give a theory for the universal curve
over the complex upper half-plane (Grojnowski's paper is now available
\cite{Grojnowski:Ell-new}).  In the case of the circle, Greenlees
\cite{MR2168575} has given a complete construction of a rational
$S^{1}$-equivariant elliptic spectrum.

Note also that Jacob Lurie has obtained analogous and sharper results
about equivariant elliptic cohomology, in the context of his derived
elliptic curves.

The functor $E_{G}$ takes its values in sheaves of
$\O_{M_{G}}$-modules, where $M_{G}$ is  the complex abelian variety
\[
       M_{G}= (\cochars\otimes_{\Z} \CQ)/W.
\]
Note that the completion of $M_{G}$ at the origin is $\spf E (BG);$ in
general one has  $E_{G} (X)^{\wedge}_{0} \iso E (EG\times_{G}X),$
where $E$ is the non-equivariant elliptic cohomology associated to
$\CQ.$

Grojnowski points out
that a construction of Looijenga \cite{Looijenga:RootSystems} (see
also \cite[\S5]{Ando:AESO}) associates to a class $c \in H^{4} (BG)$ a line bundle $\anomaly{c}$
over $M_{G}$.  Thus if $X$ is any $G$-space, then we can form the $\O_{M_{G}}$-module
\[
     E_{G} (X)_{c}\eqdef E_{G} (X)\otimes \anomaly{c}
\]
This $E_{G} (X)$-module is a twisted form of  $E_{G} (X).$

\subsection{Representations of loop groups}

Already the case of a point is interesting: one learns that twisted
equivariant elliptic cohomology carries the characters of
representations of loop groups.  Suppose that $G$ is a simple and
simply connected Lie group, such as $SU (d)$ or $Spin (2d).$  Then
\[
     H^{4} (BG)\iso \Z.
\]
We then have the following result, due independently to the first
author \cite{Ando:EllLG} (who learned it from Grojnowski) and, in a
much more precise form involving derived equivariant elliptic
cohomology, Jacob Lurie.

\begin{Proposition}
Let $G$ be a simple and simply connected compact Lie group, and let
$\dfcc \in H^{4} (BG;\Z) \iso \Z.$  The character of a representation of the
loop group $LG$ of level $\dfcc$ is a section of $\anomaly{\dfcc},$
and the
Kac character formula shows that we have an isomorphism
\[
    R_{\dfcc} (LG) \iso \Gamma (E_{G} (\ptspace)_{\dfcc})
\]
after tensoring with $\Z\lsb{q}.$
\end{Proposition}

It is fun to compare this result to the work of Freed, Hopkins,
and Teleman (for example \cite{FHT:tklgr}), who show that $R_{\dfcc}
(LG)$ is the twisted $G$-equivariant $K$-theory of $G$.  Thus we have
a map
\[
    \Gamma E_{G} (\ptspace)_{\dfcc} \to K_{G} (G)_{\dfcc}.
\]
This map is an instance of the relationship between elliptic
cohomology and the orbifold $K$-theory of the free loop space.  We
hope to provide a more extensive discussion in the future.

\subsection{The equivariant sigma orientation}

Let $\T$ be the circle group, and suppose that $V/X$ is a
$\T$-equivariant vector bundle with structure group
$G$ (in this section we suppose that $G=Spin (2d)$ or $G=SU (d)$).
Let $P/X$ be the associated principal bundle.    Then
\[
    E_{G\times \T} (P) \iso E_{\T} (X)
\]
is a sheaf of $E_{G} (\ptspace)=\O_{M_{G}}$-algebras,
and so we can twist $E_{\T} (X)$ by $\anomaly{c}$ for $c\in H^{4}
(BG)\iso \Z.$   Let $c$ be the generator corresponding to $c_{2}$ if
$G=SU (d)$ or the ``half  Pontrjagin class'' $\lambda$ if $G=Spin
(2d)$.  Note that $c$ determines a Borel equivariant class $c_{\T}.$

In \cite{Ando:AESO,AG:reso}, the authors show first of all that the
twist $E_{\T} (X)\otimes \anomaly{c}$ depends only on the equivariant
degree-four class $c_{\T} (V) \in H^{4}_{\T} (X;\Z),$
and so we may define
\[
  E_{\T} (X)_{c_{\T} (V)} = E_{\T} (X)\otimes \anomaly{c}.
\]
Second, they show that the Weierstrass
sigma function leads to an isomorphism
\begin{equation} \label{eq:52}
E_{\T} (X)_{c_{\T} (V)} \iso E_{\T} (X^{V});
\end{equation}
this is an analytic and equivariant form of the isomorphism
\eqref{eq:7}.

In the case that $c_{\T} (V) =0$ we conclude that
\begin{equation} \label{eq:5}
      E_{\T} (X)\iso E_{\T} (X^{V});
\end{equation}
this is the $\T$-equivariant sigma orientation in this context.  More
precisely,  we have the following.

\begin{Proposition}  \label{t-pr-twisted-ell-sigma-or}
Let $V/X$ be an $S^{1}$-equivariant $SU$ vector
bundle.  Let $c_{2}^{\T} (V) \in H^{4}_{\T} (X)$ be
the equivariant second Chern class of $V$.   Let $E_{\T}$ denote
Grojnowski's or Greenlees's $\T$-equivariant elliptic cohomology,
associated to the the complex  analytic elliptic curve $C.$  Then there is a
canonical isomorphism
\[
     E_{\T} (X)_{c_{2}^{\T} (V)} \iso E_{\T} (X^{V}),
\]
natural in $V/X.$  In particular if $V_{0}$ and $V_{1}$ are two such
bundles with $c_{2}^{\T} (V_{0}) = c_{2}^{\T} (V_{1})$, and
\[
    W= V_{0} - V_{1},
\]
then there is a canonical isomorphism
\[
     E_{\T} (X) \iso E_{\T}(X^{W}).
\]
\end{Proposition}

\begin{Remark}
In \cite{Ando:AESO} the author constructs the  $\T$-equivariant sigma
orientation in Grojnowski's equivariant elliptic cohomology, for
$Spin$ and $SU$ bundles.  The construction was motivated by the
Proposition stated above, which however was given as Conjecture 1.14.
In \cite{AG:reso} the authors construct the $\T$-equivariant sigma orientation
for Greenlees's equivariant elliptic cohomology, for $\T$-equivariant
$SU$-bundles.  Proposition \ref{t-pr-twisted-ell-sigma-or} appears
there as Theorem  11.17.  It should not be difficult to adapt the
methods of these two papers to the case of $Spin$ bundles.
\end{Remark}

\begin{Remark}
The careful reader will note that in \cite{AG:reso} we show how to
twist $E_{\T}^{*}(X)$ by $c_{2}^{\T} (V) \in H^{4}_{\T} (X;\Z)$: we do
not there discuss twisting by general elements of $H^{4}_{\T} (X;\Z).$
The construction of such general twists of $E_{\T} (X)$ is the subject of
on-going work of the first author and Bert Guillou.
\end{Remark}

\begin{Remark}
Lurie has obtained similar and sharper results for the elliptic
cohomology associated to a derived elliptic curve.   In particular he
can construct the sigma orientation and twists by $H^{4}_{\T}(X;\Z).$
\end{Remark}


\newcommand{\etalchar}[1]{$^{#1}$}
\def\cprime{$'$} \def\cprime{$'$}

\end{document}